\documentclass[12pt]{article}

\usepackage{amssymb}
\usepackage{pdfpages}
\usepackage{amsthm}
\usepackage{amsmath}
\usepackage{graphicx}
\usepackage{fullpage}
\usepackage{color}
\usepackage{enumerate}
 \numberwithin{equation}{section}
\usepackage{booktabs}
\allowdisplaybreaks

\usepackage[pdftex]{hyperref}
 \usepackage{mathtools}
 \usepackage{mathrsfs}
 \mathtoolsset{showonlyrefs}

\theoremstyle{plain}
\newtheorem{thm}{Theorem}[section]
\newtheorem{cor}[thm]{Corollary}

\newtheorem{lem}[thm]{Lemma}
\newtheorem{prop}[thm]{Proposition}

\newtheorem*{thm*}{Theorem}
\newtheorem*{cor*}{Corollary}
\newtheorem*{lem*}{Lemma}
\newtheorem*{defn*}{Definition}
\newtheorem*{rem*}{Remark}

\theoremstyle{definition}
\newtheorem{defn}[thm]{Definition}

\theoremstyle{remark}

\newtheorem{rem}[thm]{Remark}

\newcommand{\N}{\mathbb{N}}
\newcommand{\R}{\mathbb{R}}

\newcommand{\bp}{\begin{proof}[\ensuremath{\mathbf{Proof}}]}
\newcommand{\bs}{\begin{proof}[\ensuremath{\mathbf{Solution}}]}
\newcommand{\ep}{\end{proof}}
\newcommand{\be}{\begin{equation}}
\newcommand{\ee}{\end{equation}}

\begin{document}

\title{Sticky particles and the pressureless Euler equations \\ in one spatial dimension}

\author{Ryan Hynd\footnote{Department of Mathematics, University of Pennsylvania.  Partially supported by NSF grant DMS-1554130.}}

\maketitle

\begin{abstract}
We consider the dynamics of finite systems of point masses which move along the real line. 
We suppose the particles interact pairwise and undergo perfectly inelastic collisions when they collide.  In particular, once particles 
collide, they remain stuck together thereafter. Our main result is that if the interaction potential is semi-convex, this sticky particle property can quantified and is preserved upon letting the number of particles tend to infinity.  This is used to show that solutions of the pressureless Euler equations exist for given initial conditions and satisfy an entropy inequality. 
\end{abstract}







\section{Introduction}
In this paper, we will study solutions of the {\it pressureless Euler equations} in one spatial dimension. This is a system of partial differential equations comprised of the {\it conservation of mass}
\be\label{EP1}
\partial_t\rho +\partial_x(\rho v)=0
\ee
and the {\it conservation of momentum}
\be\label{EP2}
\partial_t(\rho v) +\partial_x(\rho v^2)=-\rho(W'*\rho).
\ee
Both equations hold in $\R\times (0,\infty)$.  This system governs the dynamics of collections of particles in one dimension whose pairwise interaction is determined by the potential $W$;  these particles also may collide and they undergo perfectly inelastic collisions when they do.  The unknowns are the density of particles $\rho$ and an associated local velocity field $v$. Our main objective in this paper is to establish the existence of solutions for given initial conditions.

\par We will suppose throughout this paper that $W: \R\rightarrow \R$ is continuously differentiable and even 
$$
W(-x)=W(x),\quad x\in \R.
$$
We note that $W$ convex in \eqref{EP2} corresponds to particles interacting via an attractive pairwise force and $W$ concave is associated with repulsive interaction.  The principal assumption made in this work is that $W$ is {\it semiconvex}. That is, there is $c>0$ such that 
\be\label{Wsemiconvex}
x\mapsto W(x)+\frac{c}{2}x^2\quad \text{is convex}.
\ee
In particular, we will study some types of interactions which are attractive and some which are repulsive.

\par  In view of the conservation of mass \eqref{EP1}, it will be natural for us to consider mass densities $\rho$ as mappings with values in the space ${\cal P}(\R)$ of Borel probability measures on $\R$.  Recall that this space has a natural topology: $(\mu^k)_{k\in \N}\subset {\cal P}(\R)$ converges to $\mu\in {\cal P}(\R)$ {\it narrowly} if 
$$
\lim_{k\rightarrow\infty}\int_{\R}gd\mu^k=\int_{\R}gd\mu
$$
for any continuous and bounded function $g:\R\rightarrow \R$.  Moreover, examples below will show that local velocities $v$ will typically be discontinuous.  However, we do expect local velocities to have reasonable integrability properties. These ideas motivate the following definition of a weak solution pair of the pressureless Euler equations.

\begin{defn}
Suppose $\rho_0\in {\cal P}(\R)$ and $v_0:\R\rightarrow \R$ is continuous. A narrowly continuous $\rho: (0,\infty)\rightarrow {\cal P}(\R); t\mapsto \rho_t$ and Borel 
$v:\R\times(0,\infty)\rightarrow\R$ is a {\it weak solution pair of the pressureless Euler equations}
which satisfies the initial conditions 
\be\label{Init}
\rho|_{t=0}=\rho_0\quad \text{and}\quad v|_{t=0}=v_0
\ee
if 
\be\label{EP1weak}
\int^\infty_0\int_{\R}(\partial_t\phi+v\partial_x\phi)d\rho_tdt+\int_{\R}\phi(\cdot,0)d\rho_0=0
\ee
and 
\be\label{EP2weak}
\int^\infty_0\int_{\R}(v\partial_t\phi +v^2\partial_x\phi)d\rho_tdt+\int_{\R}\phi(\cdot,0)v_0d\rho_0=\int^\infty_0\int_{\R}\phi(W'*\rho_t)d\rho_tdt
\ee
for each $\phi\in C^\infty_c(\R\times[0,\infty))$.
\end{defn}
\begin{rem}
Conditions \eqref{EP1weak} and \eqref{EP2weak} are weak formulations of \eqref{EP1} and \eqref{EP2}, respectively.
\end{rem}

We will construct weak solution pairs using finite particle systems. That is, we will study systems of particles with masses $m_1,\dots, m_N$ and respective trajectories $\gamma_1,\dots, \gamma_N:[0,\infty)\rightarrow \R$ that evolve in time according to Newton's second law 
\be\label{NewtonSystem}
\ddot \gamma_i(t)=-\sum^N_{j=1}m_jW'(\gamma_i(t)-\gamma_j(t)).
\ee
This system of ODE will hold at each time where there is not a collision.  When particles do collide, they experience perfectly inelastic collisions.   For example, if the subcollection of particles with masses $m_1,\dots, m_k$ collide at time $s>0$, then 
$$
m_1\dot\gamma_1(s-)+\dots +m_k\dot\gamma_k(s-)=(m_1+\dots +m_k)\dot\gamma_i(s+).
$$
for $i=1,\dots, k$. See Figure \ref{fourMass} for a schematic diagram when $k=4$.

\begin{figure}[h]
\centering
 \includegraphics[width=.75\textwidth]{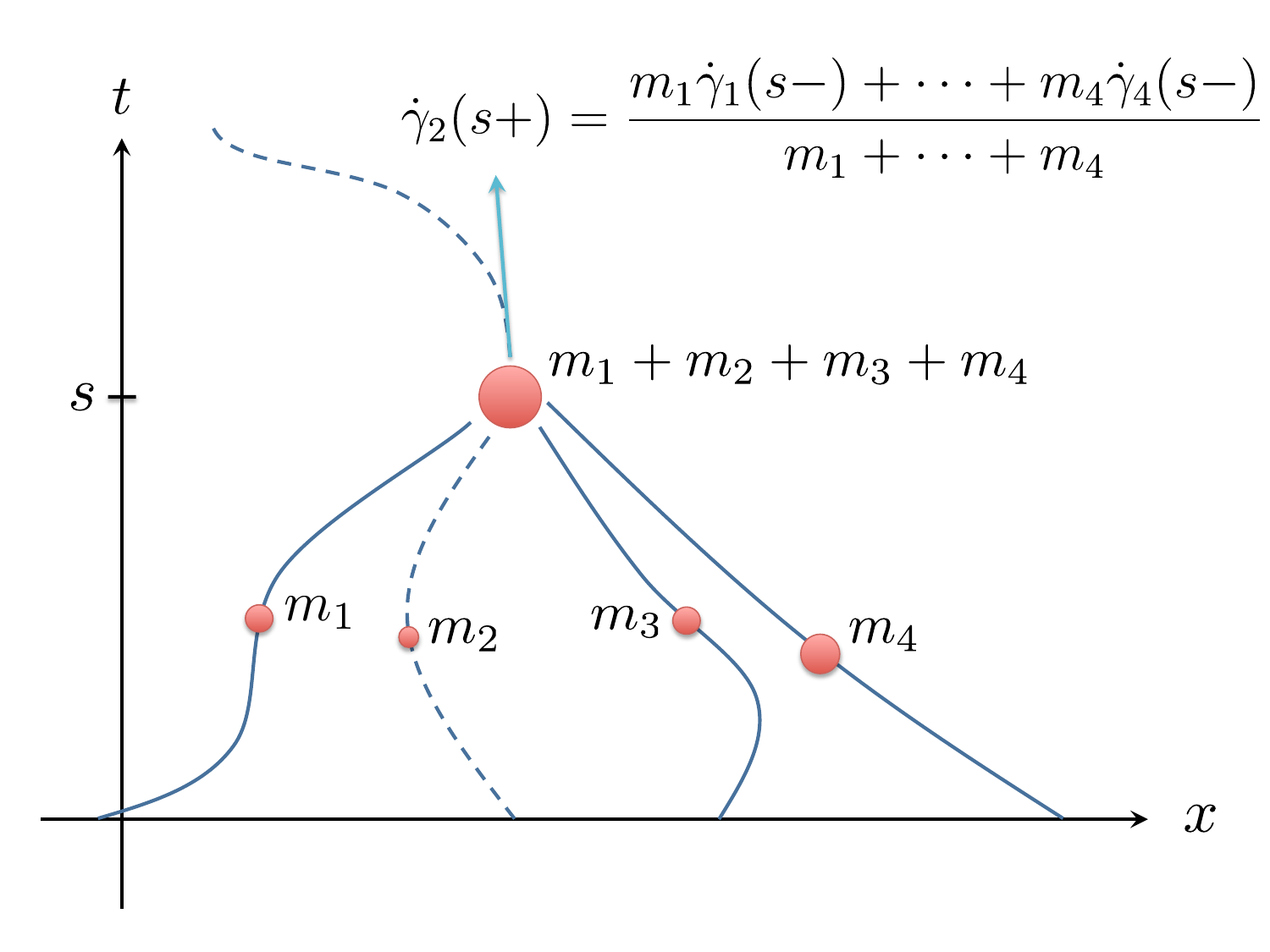}
 \caption{Point masses $m_1, m_2,m_3,$ and $m_4$ undergo a perfectly inelastic collision at time $s$. These masses are displayed larger than points to emphasize that they are allowed to be distinct.   Note in particular that the trajectories coincide after time $s$ and that the right hand limit of their slopes at time $s$ also agree. Trajectory $\gamma_2$ is shown in dashed. }\label{fourMass}
\end{figure}

\par When $\sum^N_{i=1}m_i=1$, we can define a probability measure 
$$
\rho_t:=\sum^N_{i=1}m_i\delta_{\gamma_i(t)}\in {\cal P}(\R)
$$
which represents the density of particles at time $t>0$. We can also choose a Borel function $v:\R\times(0,\infty)\rightarrow \R$ that satisfies 
$$
v(x,t)=\dot\gamma_i(t+)
$$
whenever $x=\gamma_i(t)$.  Here $v$ is a local velocity field associated to particle trajectories.  It turns out that $\rho$ and $v$ indeed comprise a weak solution pair of the pressureless Euler equations.

\par These particle trajectories satisfy 
$$
\gamma_i(s)=\gamma_j(s)\Longrightarrow \gamma_i(t)=\gamma_j(t).
$$
for $i,j=1,\dots,N$ and $s\le t$. We will actually establish the stronger {\it quantitative sticky particle property}: for $i,j=1,\dots,N$ and $0<s\le t$
\be\label{QSPP}
\frac{|\gamma_i(t)-\gamma_j(t)|}{\sinh(\sqrt{c}t)}
\le \frac{|\gamma_i(s)-\gamma_j(s)|}{\sinh(\sqrt{c}s)}.
\ee
Here $c$ is the constant in \eqref{Wsemiconvex}.  Combining \eqref{QSPP} with energy estimates derived using the semiconvexity of $W$, we will show that  the collection solutions obtained via finite particle systems are compact in a certain sense.

\par To this end, we shall assume for mathematical convenience that  
\be\label{secondMoment}
\int_{\R}x^2d\rho_0(x)<\infty
\ee
and
\be\label{Wprime}
\sup_{x\in \R}\frac{|W'(x)|}{1+|x|}<\infty.
\ee
Our main theorem is as follows.

\begin{thm}\label{mainThm} Assume $\rho_0\in {\cal P}(\R)$ satisfies \eqref{secondMoment}, $v_0:\R\rightarrow\R$ is absolutely continuous,  $W:\R\rightarrow\R$ satisfies \eqref{Wsemiconvex} for some $c>0$ and \eqref{Wprime}. There is a weak solution pair $\rho$ and $v$ of the pressureless Euler system which satisfies the initial conditions \eqref{Init}. Moreover, 
\begin{align}\label{EnergyIneq}
&\frac{1}{2}\int_\R v(x,t)^2d\rho_t(x)+\frac{1}{2}\int_\R\int_\R W(x-y)d\rho_t(x)d\rho_t(y)\\
&\hspace{1in}\le \frac{1}{2}\int_\R v(x,s)^2d\rho_s(x)+\frac{1}{2}\int_\R\int_\R W(x-y)d\rho_s(x)d\rho_s(y)
\end{align}
for almost every $0\le s\le t$; and
\be\label{entropy}
(v(x,t)-v(y,t))(x-y)\le\frac{\sqrt{c}}{\tanh(\sqrt{c}t)}(x-y)^2
\ee
for $\rho_t$ almost every $x,y\in \R$ and almost every $t>0$. 
\end{thm}

\par We note that the Euler-Poisson equations arise in a one-dimensional version of a model used by Zel'dovich \cite{Gurbatov, Zeldovich} to study the formation of large scale structures in the universe. The existence of the corresponding solution pairs for $W(x)=|x|$ was first established by E, Rykov and Sinai \cite{ERykovSinai} using a generalized variational principle.  More recently, Brenier, Gangbo, Savar\'e and Westdickenberg \cite{BreGan} conducted a general study of pressureless Euler models with attractive and repulsive interactions; in particular, they recast the pressureless Euler equations in Lagrangian coordinates and derived differential inclusions for the associated flow map. 

\par Similar approaches were used by Nguyen and Tudorascu \cite{NguTud} on the Euler-Poisson system and by Brenier and Grenier \cite{BreGre}, Natile and Savar\'e \cite{NatSav} and Cavalletti, Sedjro, and Westdickenberg \cite{MR3296602} for the sticky particle system ($ W\equiv 0$ in \eqref{EP2}). We also note that Gangbo, Nguyen, and Tudorasco have also studied the existence of solutions by exploiting the variational structure of the Euler-Poisson equations \cite{GNT}. In addition, there have been recent works on the pressureless Euler system in one spatial dimension involving the absence of shocks \cite{Guo}, hydrodynamic limits \cite{Jabin}, and entropy solutions in the presence of friction, dissipation and viscosity \cite{Shen, Jin, MR3359159}.

\par The outline of this paper is as follows.  In section \ref{StickyParticleSect}, we use the solutions of \eqref{NewtonSystem} to design sticky particle trajectories $\gamma_1,\dots,\gamma_N$ as mentioned above.  In section \ref{PathSpaceSect}, we recast weak solution pairs of \eqref{EP1} and \eqref{EP2} as probability measures on the space of continuous paths; see also \cite{Hynd} for how we treated the particular case $ W\equiv 0$. Then in section \ref{CompactSect}, we prove Theorem \ref{mainThm}. Finally, we would like to express our gratitude to Sean Paul for inquiring for a good reason as to why the entropy inequality 
\be
(v(x,t)-v(y,t))(x-y)\le\frac{1}{t}(x-y)^2
\ee
holds when $W\equiv 0$.  In trying to answer to his question, we were lead to the quantitative sticky particle property \eqref{QSPP} and subsequently to the entropy inequality \eqref{entropy} and Theorem \ref{mainThm}.

\section{Sticky particle trajectories}\label{StickyParticleSect}
In this section, we will study sticky particle trajectories $\gamma_1,\dots,\gamma_N$ as described in the introduction.  
We will show they satisfy the quantitative sticky particle property \eqref{QSPP} and also that they have important averaging 
property. This averaging property is then used to show that $\gamma_1,\dots,\gamma_N$ corresponds to a weak solution pair of 
the pressureless Euler equations.  We begin by showing these paths exist.

\begin{prop}\label{StickyParticlesExist}
Suppose $m_1,\dots, m_N>0$ with $\sum^N_{i=1}m_i=1$, $x_1\dots,x_N\in \R$ and $v_1,\dots, v_N\in \R$.   There are piecewise 
$C^2$ paths 
$$
\gamma_1,\dots,\gamma_N : [0,\infty)\rightarrow \R
$$
with the following properties. \\
(i) For $i=1,\dots, N$ and all but finitely many $t\in (0,\infty)$,
\be\label{gammaODEt}
\ddot \gamma_i(t)=-\sum^N_{j=1}m_jW'(\gamma_i(t)-\gamma_j(t)).
\ee
(ii) For $i=1,\dots, N$,
\be\label{NewtonSystemInit2}
\gamma_i(0)=x_i\quad \text{and}\quad \dot\gamma_i(0)=v_i.
\ee
(iii) For $i,j=1,\dots, N$, $0\le s\le t$ and $\gamma_i(s)=\gamma_j(s)$ imply 
$$
\gamma_i(t)=\gamma_j(t).
$$
(iv) If $t>0$, $\{i_1,\dots, i_k\}\subset\{1,\dots, N\}$, and
$$
\gamma_{i_1}(t)=\dots=\gamma_{i_k}(t)\neq \gamma_i(t)
$$
for $i\not\in\{i_1,\dots, i_k\}$, then
$$
\dot\gamma_{i_j}(t+)=\frac{m_{i_1}\dot\gamma_{i_1}(t-)+\dots+m_{i_k}\dot\gamma_{i_k}(t-)}{m_{i_1}+\dots+m_{i_k}}
$$
for $j=1,\dots, k$.
\end{prop}

\begin{proof}
We will prove the assertion by induction on $N$. Suppose $N=2$ and let $\overline\gamma_1,\overline\gamma_2$ be the solution of \eqref{gammaODEt} that satisfies \eqref{NewtonSystemInit2}; such solutions exist by Proposition \ref{ExistenceODE} in the appendix. If the trajectories $\overline\gamma_1$ and $\overline\gamma_1$ do not intersect, we take $\gamma_1=\overline\gamma_1$ and $\gamma_2=\overline\gamma_2$. Otherwise, let $s>0$ be the first time such that $z:=\overline\gamma_1(s)=\overline\gamma_2(s)$. 
We then set 
$$
\gamma_i(t):=
\begin{cases}
\overline\gamma_i(t), \quad & t\in [0,s]\\
z+(t-s)(m_1\dot{\overline\gamma_1}(s-)+m_2\dot{\overline\gamma}_2(s-)), \quad & t\in [s,\infty)
\end{cases}
$$
for $i=1,2$.  It is easily verified that $\gamma_1, \gamma_2$ satisfy $(i)-(iv)$ above. We conclude that the assertion holds for $N=2$. 

\par Now suppose the claim has been established for some $N\ge 2$.  Let $m_1,\dots, m_{N+1}>0$ with $\sum_{i}m_i=1$, $x_1\dots,x_{N+1}\in \R$, $v_1,\dots, v_{N+1}\in \R$ and assume $\overline\gamma_1,\dots, \overline\gamma_{N+1}$ is a corresponding solution of \eqref{gammaODEt} that satisfies the initial conditions \eqref{NewtonSystemInit2} (Proposition \ref{ExistenceODE}). If these trajectories never intersect, we take 
$\gamma_i=\overline\gamma_i$ for $i=1,\dots, N+1$ and conclude. Otherwise, let $s>0$ be the first time that  trajectories intersect.  First, we will assume that at time $s$ a single subcollection of trajectories $\overline\gamma_{i_1},\dots, \overline\gamma_{i_k}$ intersect. That is, 
$$
z:=\overline\gamma_{i_1}(s)=\dots=\overline\gamma_{i_k}(s)\neq \overline\gamma_i(s)
$$
for $i\not\in\{i_1,\dots, i_k\}$.   We also define  
$$
v:=\frac{m_{i_1}\dot{\overline\gamma}_{i_1}(s-)+\dots+m_{i_k}\dot{\overline\gamma}_{i_k}(s-)}{m_{i_1}+\dots+m_{i_k}}.
$$

\par By induction, there are trajectories $\{\xi_i\}_{i\neq i_j}$ and $\xi$ corresponding to the $N+1-(k-1)$ masses $\{m_i\}_{i\neq i_j}$ and $m_{i_1}+\dots +m_{i_k}$, initial positions 
$\{\overline\gamma_i(s)\}_{i\neq i_j}$ and $z$, and velocities $\{\dot{\overline\gamma_i}(s)\}_{i\neq i_j}$ and $v$ which satisfy $(i)-(iv)$ above.  We then set 
$$
\gamma_i(t):=
\begin{cases}
\overline\gamma_i(t), \quad & t\in [0,s]\\
\xi_i(t-s), \quad & t\in [s,\infty)
\end{cases}
$$
for $i\neq i_j$ and 
$$
\gamma_{i_j}(t):=
\begin{cases}
\overline\gamma_{i_j}(t), \quad &  t\in [0,s]\\
\xi(t-s), \quad & t\in [s,\infty).
\end{cases}
$$
Using the induction hypothesis, it is now routine the check that $\gamma_1,\dots,\gamma_{N+1}$ satisfy $(i)-(iv)$ in the statement of this claim. It also not difficult to see how the argument given above can be extended to the case where there are more than one subcollection of paths that intersect at time $s$.  We leave the details to the reader and conclude this assertion. 
\end{proof}
\begin{rem}\label{InterchangeDerandLimit}
The right $\dot\gamma_i(t+)$ and left $\dot\gamma_i(t-)$ limits of $\dot\gamma_i$ exist for each $t>0$ and $i=1,\dots, N$. Moreover,  
these limits can be computed as
$$
\dot\gamma_i(t\pm)=\lim_{h\rightarrow 0^\pm}\frac{\gamma_i(t+h)-\gamma_i(t)}{h}.
$$
These remarks follow from our proof of Proposition \ref{StickyParticlesExist}; they also can be established by appealing directly to property $(i)$ of the proposition. 
\end{rem}
\begin{figure}[h]
\centering
 \includegraphics[width=.75\textwidth]{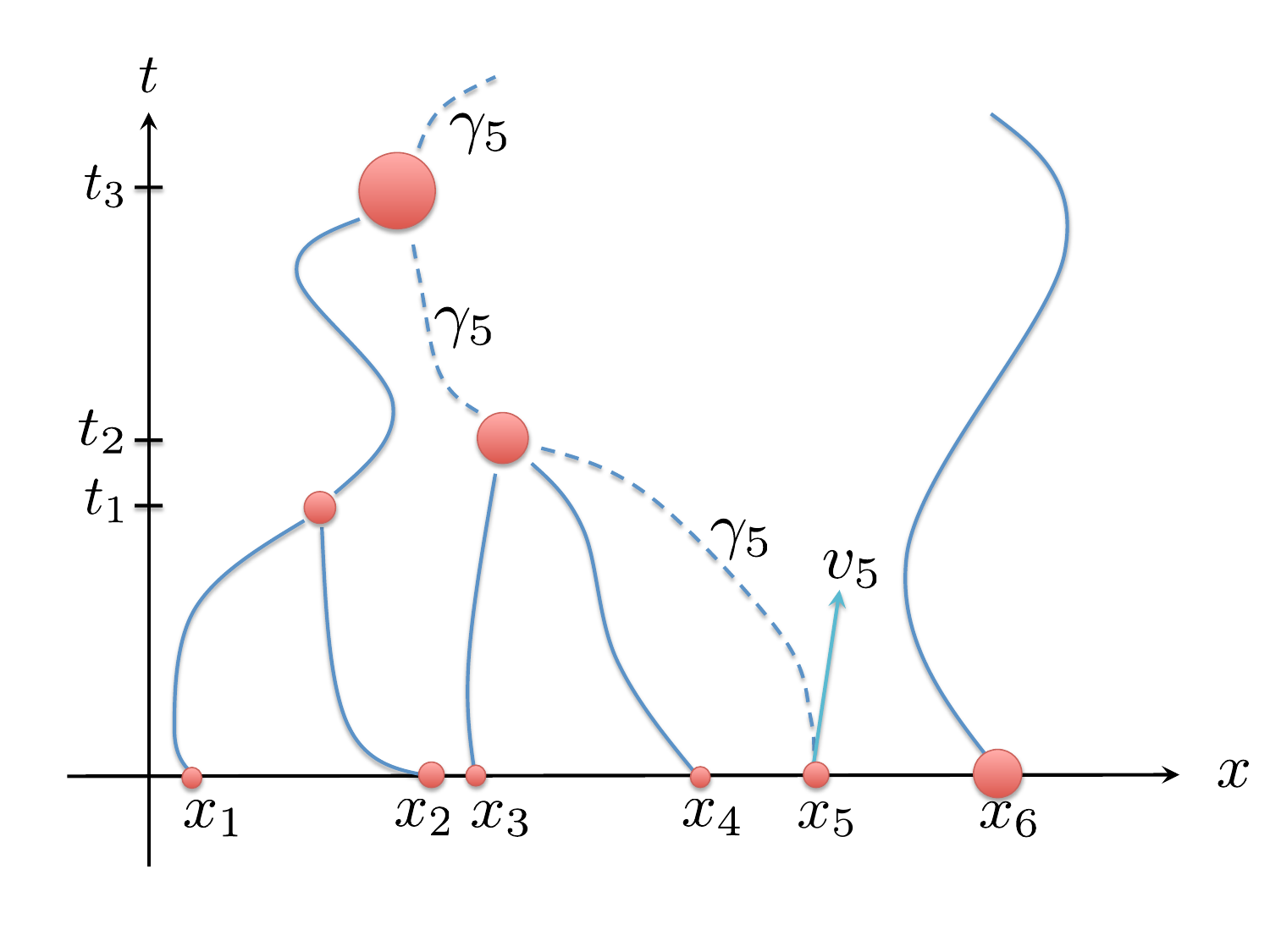}
 \caption{Sticky particle trajectories when $N=6$; here the initial positions $x_1,\dots, x_6$ are displayed along 
 with the initial velocity $v_5$ of the trajectory starting out at $x_5$. The first intersection times $t_1<t_2<t_3$ are also shown along the time axis and trajectory $\gamma_5$ is exhibited in dashed.}\label{sixTraj}
\end{figure}

\begin{defn}
The paths $\gamma_1,\dots,\gamma_{N}$ shown to exist in Proposition \ref{StickyParticlesExist} are called {\it sticky particle trajectories} corresponding to the masses $m_1,\dots, m_N$ (with $\sum_i m_i=1$), initial positions $x_1,\dots, x_N$, and initial velocities $v_1,\dots, v_N$. We also call $t>0$ a {\it first intersection time} whenever there are at least two paths $\gamma_i$ and $\gamma_j$ that agree for the first time at $t$.  See Figure \ref{sixTraj}.
\end{defn}

\subsection{Quantitative sticky particle property and stability}
For the remainder of this section, we will consider a single collection of sticky particle trajectories $\gamma_1,\dots,\gamma_{N}$  corresponding to a fixed but arbitrary collection of masses $m_1,\dots, m_N$ (with $\sum_i m_i=1$),  initial positions $x_1,\dots, x_N$, and initial velocities
$$
v_i:=v_0(x_i)
$$
where $v_0: \R\rightarrow \R$ is absolutely continuous. We will show they satisfy a quantitative version of the property $(iii)$ in Proposition \ref{StickyParticlesExist} and verify a stability property of these paths. First, we will need an elementary lemma.

\begin{lem}\label{ODEineqLemma}
Suppose $T>0$ and $y:[0,T)\rightarrow [0, \infty)$ is continuous and piecewise $C^2$. Further assume 
\be\label{slopeDecreaseY}
\dot y(t+)\le \dot y(t-)
\ee
for each $t\in (0,T)$ and that for some $c>0$
\be\label{YdoubleprimelessthanY}
\ddot y(t)\le c y(t)
\ee
for all but finitely many $t\in (0,T)$. Then $(i)$
\be
(0,T)\ni t\mapsto \frac{y(t)}{\sinh(\sqrt{c}t)}\;\; \text{is nonincreasing}
\ee
and $(ii)$
\be
y(t)\le \cosh(\sqrt{c}t)y(0)+\frac{1}{\sqrt{c}}\sinh(\sqrt{c}t)\dot y(0+),\;\;  t\in [0,T). 
\ee

\end{lem}

\begin{proof} $(i)$ Without any loss of generality, we assume $c=1$ in this proof. We will also suppose $0<t_0<\dots < t_n<T$ are times such that $y$ is $C^2$ on each of the intervals 
$(0,t_1),\dots (t_n,T)$.   It then suffices to show the continuous function
$$
u(t):=\frac{y(t)}{\sinh(t)}, \quad t\in (0,T)
$$
is nonincreasing on each of these intervals. 

\par Observe 
\begin{align*}
\dot u(t+)&=\frac{\dot y(t+)}{\sinh(t)}-\frac{y(t)}{\sinh(t)^2}\cosh(t)\\
&\le \frac{\dot y(t-)}{\sinh(t)}-\frac{y(t)}{\sinh(t)^2}\cosh(t)\\
&=\dot u(t-)
\end{align*}
for each $t>0$ by \eqref{slopeDecreaseY}. Also note 
\begin{align*}
\ddot y(t) & = \ddot u(t)\sinh(t)+2\dot u(t)\cosh(t)+u(t)\sinh(t)\\
&=\ddot u(t)\sinh(t)+2\dot u(t)\cosh(t)+y(t)\\
&\le y(t)
\end{align*}
for $t\in (0,T)\setminus\{t_1,\dots, t_n\}$. Consequently, 
\be\label{UdotLessThan}
\frac{d}{dt}\left(\dot u(t)\sinh(t)^2\right)=\sinh(t)\left(\ddot u(t)\sinh(t)+2\dot u(t)\cosh(t)\right)\le 0
\ee
for $t\in (0,T)\setminus\{t_1,\dots, t_n\}$. 

\par 
As $y$ is nonnegative, 
$$
\dot u(t)\le \frac{\dot y(t)}{\sinh(t)}.
$$
for $t\in (0,T)\setminus\{t_1,\dots, t_n\}$.  As a result, 
$$
\limsup_{s\rightarrow 0^+}\left\{\dot u(s)\sinh(s)^2\right\}\le \limsup_{s\rightarrow 0^+}\left\{\dot y(s)\sinh(s)\right\}=\dot y(0+)\sinh(0)=0.
$$
In view of \eqref{UdotLessThan}
$$
\dot u(t)\sinh(t)^2\le \limsup_{s\rightarrow 0^+}\left\{\dot u(s)\sinh(s)^2\right\}=0
$$
for $t\in (0,t_1)$; we emphasize that this inequality is valid since $\dot u$ is $C^1$ on $(0,t_1)$ which allows us to integrate the left hand side of \eqref{UdotLessThan} on this interval.  Therefore, $\dot u(t)\le 0$ for $t\in (0,t_1)$. 

\par So far, we have that $u$ is nonincreasing on $t\in (0,t_1]$ and
$$
\dot u(t_1-)\le 0.
$$
In addition, \eqref{UdotLessThan} gives
\begin{align*}
\dot u(t)\sinh(t)^2&\le \dot u(t_1+)\sinh(t_1)^2\\
&\le \dot u(t_1-)\sinh(t_1)^2\\
&\le 0
\end{align*}
for $t\in (t_1,t_2)$.  Thus, $u$ is nonincreasing on $[t_1,t_2]$ and 
$$
\dot u(t_2-)\le 0.
$$
It is now evident that we may repeat this argument to show that $u$ is nonincreasing on $[t_2,t_3], [t_3,t_4],\dots, [t_n,T)$ and 
therefore on $(0,T)$.

\par Part $(ii)$ of this lemma follows similarly. We will omit the analogous argument this claim has been established in Lemma 3.7 of \cite{Hynd2}. 
\end{proof}

\par We now verify the following quantitative sticky particle property and stability estimate of sticky particle trajectories. We recall that $W(x)+(c/2)x^2$ is convex for some $c>0$.
\begin{prop}\label{PropQSPP}
Suppose $i,j\in\{1,\dots,N\}$. \\
$(i)$ For $0<s\le t$,
\be
\frac{|\gamma_i(t)-\gamma_j(t)|}{\sinh(\sqrt{c}t)}
\le \frac{|\gamma_i(s)-\gamma_j(s)|}{\sinh(\sqrt{c}s)}.
\ee
$(ii)$ If $x_i\ge x_j$, then 
\be\label{timeZeroEst}
0\le \gamma_i(t)-\gamma_j(t)\le \cosh(\sqrt{c}t)(x_i-x_j)+\frac{1}{\sqrt{c}}\sinh(\sqrt{c}t)\int^{x_i}_{x_j}|v_0'(x)|dx
\ee
for all $t\ge 0$.  
\end{prop}
\begin{proof}
Without any loss of generality, we may assume $\gamma_1\le\dots\le \gamma_N$. It then suffices to prove this proposition for $i=1,\dots, N-1$ and $j=i+1$.  Indeed, if
\be\label{neighborsGammaiplusone}
\frac{\gamma_{i+1}(t)-\gamma_i(t)}{\sinh(\sqrt{c}t)}\quad \text{is nonincreasing}
\ee
for $i=1,\dots, N-1$, then
\be
\frac{\gamma_k(t)-\gamma_j(t)}{\sinh(\sqrt{c}t)}
=\sum^{k-1}_{i=j}\frac{\gamma_{i+1}(t)-\gamma_i(t)}{\sinh(\sqrt{c}t)}
\ee
is a sum of nonincreasing functions for $k>j$ which would prove assertion $(i)$. 

Likewise, for assertion $(ii)$, it suffices to verify 
\be\label{timeZeroEstiplusone}
\gamma_{i+1}(t)-\gamma_i(t)\le \cosh(\sqrt{c}t)(x_{i+1}-x_i)+\frac{1}{\sqrt{c}}\sinh(\sqrt{c}t)|v_0(x_{i+1})-v_0(x_i)|
\ee
for $t\ge 0$. In this case, 
\begin{align}
\gamma_k(t)-\gamma_j(t)&= \sum^{k-1}_{i=j}(\gamma_{i+1}(t)-\gamma_i(t))\\
&\le \sum^{k-1}_{i=j}\left(\cosh(\sqrt{c}t)(x_{i+1}-x_i)+\frac{1}{\sqrt{c}}\sinh(\sqrt{c}t)
|v_0(x_{i+1})-v_0(x_i)|\right)\\
&\le  \cosh(\sqrt{c}t)(x_k-x_j)+\frac{1}{\sqrt{c}}\sinh(\sqrt{c}t)\sum^{k-1}_{i=j}\int^{x_{i+1}}_{x_i}|v_0'(x)|dx\\
&= \cosh(\sqrt{c}t)(x_k-x_j)+\frac{1}{\sqrt{c}}\sinh(\sqrt{c}t)\int^{x_{k}}_{x_j}|v_0'(x)|dx.
\end{align}
 
\par Consequently, we will fix $i \in \{1,\dots, N-1\}$ and focus on establishing \eqref{neighborsGammaiplusone} and 
\eqref{timeZeroEstiplusone}.   Our plan is to verify the hypotheses of Lemma \ref{ODEineqLemma} with 
$$
y(t):=\gamma_{i+1}(t)-\gamma_i(t), \quad t\in [0,T),
$$
where
$$
T:=\inf\{t> 0: \gamma_{i+1}(t)=\gamma_i(t)\}.
$$
Of course if $\gamma_{i+1}(t)>\gamma_i(t)$ for all $t>0$, then $T=+\infty$. Otherwise, we have
$$
\gamma_{i+1}(t)-\gamma_i(t)=0
$$
for $t\ge T$. In either case, it is enough to prove  \eqref{neighborsGammaiplusone} and \eqref{timeZeroEstiplusone} on $[0,T)$.  To this end, we will first show
\be\label{DotGmamaiplusone}
\dot\gamma_{i+1}(s+)\le \dot\gamma_{i+1}(s-)
\ee
and 
\be\label{DotGmamajusti}
\dot\gamma_{i}(s+)\ge \dot\gamma_{i}(s-)
\ee
for each $s\in (0,T)$. 

\par  We will only justify \eqref{DotGmamaiplusone} as \eqref{DotGmamajusti} can be proved similarly. Observe that if $\gamma_{i+1}$ does not have a first intersection time at $s\in (0,T)$, then $\gamma_{i+1}$ is $C^1$ in a neighborhood of $s$ and thus 
$$
\dot\gamma_{i+1}(s)=\dot\gamma_{i+1}(s+)=\dot\gamma_{i+1}(s-).
$$
If $\gamma_{i+1}$ has a first intersection time at $s\in (0,T)$, then there are trajectories $\gamma_{i+2}, \dots, \gamma_{i+r}$  (some $r\ge 2$) such that 
$$
\gamma_{i+1}(s)=\gamma_{i+2}(s)=\dots= \gamma_{i+r}(s)
$$
and 
\be\label{AveragingForQSPS}
\dot\gamma_{i+j}(s+)=\frac{m_{i+1}\dot\gamma_{i+1}(s-)+\dots + m_{i+r}\dot\gamma_{i+r}(s-)}{m_{i+1}+\dots+m_{i+r}}
\ee
$j=1,\dots, r$. 

\par Also note that as $\gamma_{i+1}\le \gamma_{i+j}$ for $j=2,\dots, r$,
$$
\frac{\gamma_{i+1}(s+h)- \gamma_{i+1}(s)}{h}\ge \frac{\gamma_{i+j}(s+h)- \gamma_{i+j}(s)}{h}
$$
for all $h<0$ and sufficiently small. By Remark \ref{InterchangeDerandLimit}, we can send $h\rightarrow 0^-$ to find
$$
\dot\gamma_{i+1}(s-)\ge \dot\gamma_{i+j}(s-)
$$
for $j=2,\dots, r$. It then follows from \eqref{AveragingForQSPS} that 
 $$
\dot \gamma_{i+1}(s+)\le\frac{m_{i+1}\dot\gamma_{i+1}(s-)+\dots + m_{i+r}\dot\gamma_{i+1}(s-)}{m_{i+1}+\dots+m_{i+r}}=\dot\gamma_{i+1}(s-),
 $$
  which is \eqref{DotGmamaiplusone}.  Putting  \eqref{DotGmamaiplusone} and  \eqref{DotGmamajusti} together gives 
   \be\label{dotYnotgoingup}
 \dot y(s+)=\dot\gamma_{i+1}(s+)-\dot\gamma_{i}(s+)\le \dot\gamma_{i+1}(s-)-\dot\gamma_{i}(s-)= \dot y(s-)
 \ee
for all $s\in (0,T)$. 
  
\par By part $(i)$ of Proposition \ref{StickyParticlesExist} and the semiconvexity of $W$,
\begin{align*}
\ddot y(t)&=\ddot\gamma_{i+1}(t)-\ddot\gamma_i(t)\\
&=-\sum_{j=1}^Nm_j\left(W'(\gamma_{i+1}(t)-\gamma_j(t))-W'(\gamma_i(t)-\gamma_j(t))\right) \\
&\le \sum_{j=1}^Nm_j c\left(\gamma_{i+1}(t)-\gamma_i(t)\right)\\
&= c(\gamma_{i+1}(t)-\gamma_i(t))\\
&=c y(t)
\end{align*}
for all but finitely many $t>0$.  Combining this observation with \eqref{dotYnotgoingup} allows us to apply Lemma \ref{ODEineqLemma} and conclude.
\end{proof}

\subsection{Averaging property}
We will now discuss the averaging property of the paths $\gamma_1,\dots, \gamma_N$. We shall see that it implies a statement about the conservation of momentum of collections of finitely many sticky particles.

\begin{prop}
Suppose $g:\R\rightarrow \R$ and $0\le s<t$. Then 
\begin{align}\label{AveragingProp}
&\sum^N_{i=1}m_ig(\gamma_i(t))\dot\gamma_i(t+)= \\
&\quad \sum^N_{i=1}m_ig(\gamma_i(t))\left[\dot\gamma_i(s+)-\int^t_s\left(\sum^N_{j=1}m_jW'(\gamma_i(\tau)-\gamma_j(\tau))\right)d\tau \right].
\end{align}
\end{prop}
\begin{proof}
Let $t_0=0$ and $0<t_1<\dots<t_\ell$ denote the first intersection times of the paths $\gamma_1,\dots, \gamma_N$.  As these paths satisfy the ODE \eqref{gammaODEt} on $(0,\infty)\setminus\{t_1,\dots, t_\ell\}$, 
it suffices to verify 
\begin{align}\label{AveragingPropReduced}
&\sum^N_{i=1}m_ig(\gamma_i(t))\dot\gamma_i(t_r+)= \\
&\quad \sum^N_{i=1}m_ig(\gamma_i(t))\left[\dot\gamma_i(t_k+)-\int^{t_r}_{t_k}\left(\sum^N_{j=1}m_jW'(\gamma_i(\tau)-\gamma_j(\tau))\right)d\tau \right]
\end{align}
for each $t_k\le t_r$, where $t_r$ is the largest element of $\{t_0,t_1,\dots, t_\ell\}$ that is less than or equal to $t$.  We will establish the identity \eqref{AveragingPropReduced} by induction on $k\le r$.  Of course \eqref{AveragingPropReduced} is clear for $k=r$. So we assume it holds for some $k<r$ and
then show it holds for $k-1$. 

\par At time $t_k$ let us initially suppose that a single subcollection $\gamma_{i_1},\dots, \gamma_{i_n}$ of paths intersect for the first time. This implies
$$
\gamma_{i_1}(t_k)=\dots=\gamma_{i_n}(t_k)\neq \gamma_i(t_k)
$$
for each $i\not\in \{i_1,\dots, i_n\}$, $\overline{\gamma}(t):=\gamma_{i_1}(t)=\dots=\gamma_{i_n}(t)$ since $t\ge t_k$, and also that
\be\label{DerGammEyePeeEq}
\dot\gamma_{i_p}(t_k+)=\frac{m_{i_1}\dot\gamma_{i_1}(t_k-)+\dots +m_{i_n}\dot\gamma_{i_n}(t_k-)}{m_{i_1}+\dots +m_{i_n}}
\ee
$p=1,\dots, n$.   Furthermore, when  $i\not\in \{i_1,\dots, i_n\}$, $\gamma_i$ is continuously differentiable on $(t_{k-1},t_{k+1})$ if $k<\ell$ or on $(t_{k-1},\infty)$ if $k=\ell$. 

\par With these observations and the induction hypothesis, we have
\begin{align}
&\sum^N_{i=1}m_ig(\gamma_i(t))\dot\gamma_i(t_r+) \\
&=\quad \sum^N_{i=1}m_ig(\gamma_i(t))\left[\dot\gamma_i(t_k+)-\int^{t_r}_{t_k}\left(\sum^N_{j=1}m_jW'(\gamma_i(\tau)-\gamma_j(\tau))\right)d\tau \right]\\
&=\quad \sum_{i\not\in \{i_1,\dots, i_n\}}m_ig(\gamma_i(t))\left[\dot\gamma_i(t_k+)-\int^{t_r}_{t_k}\left(\sum^N_{j=1}m_jW'(\gamma_i(\tau)-\gamma_j(\tau))\right)d\tau \right]\\
&+\quad \sum^n_{p=1}m_{i_p}g(\gamma_{i_p}(t))\left[\dot\gamma_{i_p}(t_k+)-\int^{t_r}_{t_k}\left(\sum^N_{j=1}m_jW'(\gamma_{i_p}(\tau)-\gamma_j(\tau))\right)d\tau \right]\\
&=\quad \sum_{i\not\in \{i_1,\dots, i_n\}}m_ig(\gamma_i(t))\left[\dot\gamma_i(t_{k-1}+)-\int^{t_r}_{t_{k-1}}\left(\sum^N_{j=1}m_jW'(\gamma_i(\tau)-\gamma_j(\tau))\right)d\tau \right]\\
&+\quad \sum^n_{p=1}m_{i_p}g(\overline\gamma(t))\left[\dot\gamma_{i_p}(t_k+)-\int^{t_r}_{t_k}\left(\sum^N_{j=1}m_jW'(\overline\gamma(\tau)-\gamma_j(\tau))\right)d\tau \right]\\
&=\quad \sum_{i\not\in \{i_1,\dots, i_n\}}m_ig(\gamma_i(t))\left[\dot\gamma_i(t_{k-1}+)-\int^{t_r}_{t_{k-1}}\left(\sum^N_{j=1}m_jW'(\gamma_i(\tau)-\gamma_j(\tau))\right)d\tau \right]\\
&+\quad \sum^n_{p=1}m_{i_p}g(\overline\gamma(t))\left[\dot\gamma_{i_p}(t_k-)-\int^{t_r}_{t_k}\left(\sum^N_{j=1}m_jW'(\overline\gamma(\tau)-\gamma_j(\tau))\right)d\tau \right]\\
&=\quad \sum_{i\not\in \{i_1,\dots, i_n\}}m_ig(\gamma_i(t))\left[\dot\gamma_i(t_{k-1}+)-\int^{t_r}_{t_{k-1}}\left(\sum^N_{j=1}m_jW'(\gamma_i(\tau)-\gamma_j(\tau))\right)d\tau \right]\\
&+\quad \sum^n_{p=1}m_{i_p}g(\overline\gamma(t))\left[\dot\gamma_{i_p}(t_{k-1}+)-\int^{t_r}_{t_{k-1}}\left(\sum^N_{j=1}m_jW'(\overline\gamma(\tau)-\gamma_j(\tau))\right)d\tau \right]\\
&=\quad \sum^N_{i=1}m_ig(\gamma_i(t))\left[\dot\gamma_i(t_{k-1}+)-\int^{t_r}_{t_{k-1}}\left(\sum^N_{j=1}m_jW'(\gamma_i(\tau)-\gamma_j(\tau))\right)d\tau \right].
\end{align}
Note that we used \eqref{DerGammEyePeeEq} to derive the fourth equality above. Finally, we note that if more than one subcollection of trajectories intersect for the first time at $t_k$, we can argue as above on each subcollection to verify \eqref{AveragingPropReduced}. Therefore, the conclusion follows 
by induction. 
\end{proof}
\begin{rem}\label{AveragingRemark}
Choosing 
$$
g(x)=
\begin{cases}
1,\quad &x=\gamma_j(t)\\
0,\quad &x\neq \gamma_j(t)
\end{cases}
$$
in \eqref{AveragingProp} and sending $s\rightarrow t^-$ gives
\be\label{SimpleAveragingProp}
\left(\sum_{\gamma_i(t)=\gamma_j(t)}m_i\right)\dot\gamma_j(t+)=\sum_{\gamma_i(t)=\gamma_j(t)}m_i\dot\gamma_i(t-)\ee
for $j=1,\dots, N$. Each summation above is taken over $i\in \{1,\dots, N\}$ such that $\gamma_i(t)=\gamma_j(t)$.
\end{rem}

\subsection{Energy estimates}
We also can prove that the total energy of finite particle systems is non-increasing in time.  In particular, the total energy will only be constant for systems where 
particles do not collide. 
\begin{prop}
For each $0\le s< t$
\begin{align}\label{ContNonincreaseEnergy}
&\frac{1}{2}\sum^N_{i=1}m_i\dot\gamma_i(t+)^2+
\frac{1}{2}\sum^N_{i,j=1}m_im_jW(\gamma_i(t)-\gamma_j(t))\\
&\quad \le \frac{1}{2}\sum^N_{i=1}m_i\dot\gamma_i(s+)^2+
\frac{1}{2}\sum^N_{i,j=1}m_im_jW(\gamma_i(s)-\gamma_j(s)).
\end{align}
\end{prop}
\begin{proof}
Let $t_0=0$ and $t_1<\dots<t_\ell$ denote the first intersection times of the paths $\gamma_1,\dots, \gamma_N$. Recall that $\gamma_1,\dots, \gamma_N$ satisfy the ODE \eqref{gammaODEt} on $(t_{k-1},t_k)$ for $k=1,\dots, \ell$
and so the conservation of energy holds on these intervals. And in view of Remark \ref{AveragingRemark}, we can apply Jensen's inequality to derive
$$
\frac{1}{2}\sum^N_{i=1}m_i\dot\gamma_i(t_{k}+)^2\le \frac{1}{2}\sum^N_{i=1}m_i\dot\gamma_i(t_{k}-)^2
$$
for $k=1,\dots, \ell$. Consequently, 
\begin{align}\label{DisNonincreaseEnergy}
&\frac{1}{2}\sum^N_{i=1}m_i\dot\gamma_i(t_{k+1}+)^2+
\frac{1}{2}\sum^N_{i,j=1}m_im_jW(\gamma_i(t_{k+1})-\gamma_j(t_{k+1})) \nonumber \nonumber \\
&\quad\le\frac{1}{2}\sum^N_{i=1}m_i\dot\gamma_i(t_{k+1}-)^2+
\frac{1}{2}\sum^N_{i,j=1}m_im_jW(\gamma_i(t_{k+1})-\gamma_j(t_{k+1}))\nonumber \\
&\quad=\frac{1}{2}\sum^N_{i=1}m_i\dot\gamma_i(t_{k}+)^2+
\frac{1}{2}\sum^N_{i,j=1}m_im_jW(\gamma_i(t_{k})-\gamma_j(t_{k}))
\end{align}
for $k=0,\dots, \ell-1.$

\par Now suppose $0\le s<t$. If no $t_1,\dots,t_\ell$ belong to the interval $(s,t)$, we conclude by the conservation of 
energy. Otherwise, select $t_k$ to be the smallest $\{t_1,\dots, t_\ell\}$ belonging to $(s,t)$ and select
$t_r$ to be the largest $\{t_1,\dots, t_\ell\}$ belonging to $(s,t)$. Then $\gamma_1,\dots, \gamma_N$ satisfies
\eqref{gammaODEt} on $(s,t_k)$ and on $(t_r, t)$ so that the conservation holds on these intervals. Combining with  
\eqref{DisNonincreaseEnergy} then gives
\begin{align}
&\frac{1}{2}\sum^N_{i=1}m_i\dot\gamma_i(t+)^2+
\frac{1}{2}\sum^N_{i,j=1}m_im_jW(\gamma_i(t)-\gamma_j(t)) \nonumber \nonumber \\
&\quad=\frac{1}{2}\sum^N_{i=1}m_i\dot\gamma_i(t_{r}+)^2+
\frac{1}{2}\sum^N_{i,j=1}m_im_jW(\gamma_i(t_{r})-\gamma_j(t_{r}))\nonumber \\
&\quad\le \frac{1}{2}\sum^N_{i=1}m_i\dot\gamma_i(t_{k}+)^2+
\frac{1}{2}\sum^N_{i,j=1}m_im_jW(\gamma_i(t_{k})-\gamma_j(t_{k}))\\
&\quad\le \frac{1}{2}\sum^N_{i=1}m_i\dot\gamma_i(t_{k}-)^2+
\frac{1}{2}\sum^N_{i,j=1}m_im_jW(\gamma_i(t_{k})-\gamma_j(t_{k}))\\
 &\quad=\frac{1}{2}\sum^N_{i=1}m_i\dot\gamma_i(s+)^2+
\frac{1}{2}\sum^N_{i,j=1}m_im_jW(\gamma_i(s)-\gamma_j(s)).
\end{align}
\end{proof}
\begin{cor}\label{discreteEnergyEstimates}  
Define 
\be\label{varphiFun}
\varphi(t):=e^{(c+1)t^2}\int^t_{0}e^{-(c+1)s^2}ds.
\ee
For each $0\le t_1\le t_2$,  
\be\label{DiscreteEnergyEstA}
\int^{t_2}_{t_1}\sum^N_{i=1}m_i\dot\gamma_i(s)^2ds\le (\varphi(t_2)-\varphi(t_1))\;
\left(\sum^N_{i=1}m_iv_0(x_i)^2 +\frac{1}{2}\sum^N_{i,j=1}m_im_jW'(x_i-x_j)^2 \right).
\ee
\end{cor}
\begin{proof}
As $x\mapsto W(x)+(c/2)x^2$ is convex,
\begin{align*}
W(\gamma_i(t)-\gamma_j(t))&\ge W(x_i-x_j)+ W'(x_i-x_j)(\gamma_i(t)-\gamma_j(t)-(x_i-x_j)) \\
&\quad\quad\quad-\frac{c}{2}(\gamma_i(t)-x_i-(\gamma_j(t)-x_j))^2 \\ 
&\ge W(x_i-x_j)- \frac{1}{2}W'(x_i-x_j)^2  -\frac{c+1}{2}(\gamma_i(t)-x_i-(\gamma_j(t)-x_j))^2 \\
&\ge W(x_i-x_j)- \frac{1}{2}W'(x_i-x_j)^2  -(c+1)((\gamma_i(t)-x_i)^2+(\gamma_j(t)-x_j)^2) \\
&\ge W(x_i-x_j)- \frac{1}{2}W'(x_i-x_j)^2 -(c+1)t\left( \int^t_{0}\dot\gamma_i(s)^2ds+\int^t_{0}\dot\gamma_j(s)^2ds \right).
\end{align*}
Combining this lower bounds with \eqref{ContNonincreaseEnergy} gives
\begin{align}\label{SteptoSecondGronwall}
\sum^N_{i=1}m_i\dot\gamma_i(t)^2&\le \sum^N_{i=1}m_iv_0(x_i)^2 +\frac{1}{2}\sum^N_{i,j=1}m_im_jW'(x_i-x_j)^2 \nonumber\\
&\hspace{1in}+ 2(c+1)t\int^t_{0}\sum^N_{i=1}m_i\dot\gamma_i(s)^2ds.
\end{align}
\par As a result,
\begin{align*}
&\frac{d}{dt}e^{-(c+1)t^2}\int^t_{0}\sum^N_{i=1}m_i\dot\gamma_i(s)^2ds\\
&\quad =e^{-(c+1)t^2}\left(\sum^N_{i=1}m_i\dot\gamma_i(t)^2-2(c+1) t\int^t_{0}\sum^N_{i=1}m_i\dot\gamma_i(s)^2ds\right)\\
&\quad \le e^{-(c+1)t^2}\left(\sum^N_{i=1}m_iv_0(x_i)^2 +\frac{1}{2}\sum^N_{i,j=1}m_im_jW'(x_i-x_j)^2\right).
\end{align*}
Integrating from $0$ to $t$ gives,
$$
\int^t_{0}\sum^N_{i=1}m_i\dot\gamma_i(s)^2ds\le \varphi(t)\left(\sum^N_{i=1}m_iv_0(x_i)^2 +\frac{1}{2}\sum^N_{i,j=1}m_im_jW'(x_i-x_j)^2 \right).
$$
Upon substituting this inequality in \eqref{SteptoSecondGronwall}, we find 
\begin{align*}
\sum^N_{i=1}m_i\dot\gamma_i(t)^2&\le (1+\kappa t \varphi(t))\left(\sum^N_{i=1}m_iv_0(x_i)^2 +\frac{1}{2}\sum^N_{i,j=1}m_im_jW'(x_i-x_j)^2 \right)\\
&= \varphi'(t) \left(\sum^N_{i=1}m_iv_0(x_i)^2+\frac{1}{2}\sum^N_{i,j=1}m_im_jW'(x_i-x_j)^2 \right).
\end{align*}
Inequality \eqref{DiscreteEnergyEstA} now follows from integrating from $t_1$ to $t_2$. 
\end{proof}

\section{Probability measures on the path space}\label{PathSpaceSect}
We now consider $\Gamma:=C([0,\infty))$, the space of continuous paths from $[0,\infty)$ into $\R$, equipped with the following distance 
$$
d(\gamma,\xi):=\sum_{n\in \N}\frac{1}{2^n}\left(\frac{\displaystyle\max_{0\le t\le n}|\gamma(t)-\xi(t)|}{1+\displaystyle\max_{0\le t\le n}|\gamma(t)-\xi(t)|}\right)
\quad\quad (\gamma,\xi\in \Gamma).
$$
It is routine to check that $\lim_{k\rightarrow\infty}d(\xi_k,\xi)= 0$ if and only if $\xi_k\rightarrow \xi$ locally uniformly on $[0,\infty)$. It is also not difficult 
to verify that $\Gamma$ is a complete and separable metric space (see for instance the appendix of \cite{MR3302526}).

\par In this section, we will associate a Borel probability measure on $\Gamma$, which we will write as $\eta\in {\cal P}(\Gamma)$, to a given $\rho_0\in {\cal P}(\R)$ and absolutely continuous $v_0: \R\rightarrow \R$.  In particular, we will interpret 
the support of $\eta$ as the set of trajectories of a collection of evolving point masses which which interact pairwise with potential $W$ and via perfectly inelastic collisions when collisions occur.   To this end, we will employ the evaluation map 
$$
e_t: \Gamma\rightarrow \R; \gamma\mapsto \gamma(t)
$$ 
and the push forward measure $e_t{_\#}\eta\in {\cal P}(\R)$
\be\label{PushForward}
\int_{\R}gd(e_t{_\#}\eta):=\int_{\Gamma}g(\gamma(t))d\eta(\gamma),
\ee
for each $t\ge 0$.  

\par We will also make use of the space
$$
X:=\left\{\gamma: [0,\infty)\rightarrow \R: \text{$\gamma$ absolutely continuous,}\; \int^n_0\dot\gamma(t)^2dt<\infty \;\text{for all $n\in \N$}\right\}
$$
and the function
\be\label{Gfunction}
\Psi(\gamma)
:=
\begin{cases}
\displaystyle\sum^\infty_{n=1}\frac{1}{2^n\varphi(n)}\left(\int^n_0\dot\gamma(t)^2dt+\gamma(0)^2\right), \quad & \gamma\in X\\
+\infty, \quad & \gamma\not\in X
\end{cases}
\ee
($\gamma\in \Gamma$).  Recall that $\varphi$ was defined in \eqref{varphiFun}. It is a straightforward exercise to employ the Arzel\`a-Ascoli theorem and check that $\Psi$ has compact sublevel sets within $\Gamma$.  Our central existence assertion is as follows.

\begin{thm}\label{EtaThm}
Assume $\rho_0\in {\cal P}(\R)$ satisfies \eqref{secondMoment} and $v_0:\R\rightarrow\R$ is absolutely continuous.  Further suppose $W:\R\rightarrow\R$ satisfies \eqref{Wsemiconvex} for some $c>0$ and \eqref{Wprime}.
There is $\eta\in {\cal P}(\Gamma)$ which has the following properties. 
\begin{enumerate}[(i)]

\item  For $\eta$ almost every $\gamma\in \Gamma$, $\Psi(\gamma)<\infty$. 

\item $\rho_0=e_0{_\#}\eta$.

\item For each $0<s\le t$ and $\gamma,\xi\in \textup{supp}(\eta)$,  
\be
\frac{|\gamma(t)-\xi(t)|}{\sinh(\sqrt{c}t)}\le \frac{|\gamma(s)-\xi(s)|}{\sinh(\sqrt{c}s)}.
\ee

\item For each $t \ge 0$ and $\gamma,\xi\in \textup{supp}(\eta)$ with $\gamma(0)\ge \xi(0)$,  
\be
0\le \gamma(t)-\xi(t)\le \cosh(\sqrt{c}t)(\gamma(0)-\xi(0))+\frac{1}{\sqrt{c}}\sinh(\sqrt{c}t)\int^{\gamma(0)}_{\xi(0)}|v_0'(x)|dx
\ee

\item There is a Borel $v:\R\times(0,\infty)\rightarrow \R$ such that
$$
\dot\gamma(t)=v(\gamma(t),t)\;\; \text{a.e.}\; t>0
$$
for $\eta$ almost every $\gamma\in \Gamma$. 

\item For  almost every $t\ge 0$ and each Borel $h:\R\rightarrow \R$ with $\int_{\Gamma}h(\gamma(t))^2d\eta(\gamma)<\infty$,
\be\label{EtaConsMom}
\int_{\Gamma}\dot\gamma(t)h(\gamma(t))d\eta(\gamma)=\int_{\Gamma}\left(v_0(\gamma(0)) -\int^t_0W'*(e_s{_\#}\eta)(\gamma(s))ds\right)h(\gamma(t))d\eta(\gamma).
\ee
 
\item For  almost every $0\le s\le t$,
\begin{align*}
&\int_{\Gamma}\frac{1}{2}\dot\gamma(t)^2d\eta(\gamma)+\frac{1}{2}\int_{\Gamma}\int_{\Gamma}W(\gamma(t)-\xi(t))d\eta(\gamma)d\eta(\xi)\\
&\hspace{1in} \le \int_{\Gamma}\frac{1}{2}\dot\gamma(s)^2d\eta(\gamma)+\frac{1}{2}\int_{\Gamma}\int_{\Gamma}W(\gamma(s)-\xi(s))d\eta(\gamma)d\eta(\xi).
\end{align*}

\end{enumerate}

\end{thm}
Our first step in proving this theorem is showing that it holds when $\rho_0$ is a convex combination of Dirac measures. In this case, we also establish two key estimates.  In proving estimate \eqref{PointwiseBound} below, we will recall that since $v_0$ is absolutely continuous,   
$$
\omega(r):=\sup\left\{\int^b_a|v_0'(x)|dx: 0\le b-a\le r\right\},\quad r\ge 0
$$ 
tends to $0$ as $r\rightarrow 0^+$. In particular, $\omega$ is uniformly continuous and grows at most linearly; so may choose $\alpha>0$ for which 
\be\label{OmegaAlpha}
\omega(r)\le \alpha(r+1),\quad r\ge 0. 
\ee
\begin{lem}\label{EtaLem}
Suppose 
$$
\rho_0=\sum^N_{i=1}m_i\delta_{x_i}\in {\cal P}(\R)
$$
and let $\gamma_1,\dots, \gamma_N$ be a collection of sticky particle trajectories with masses $m_1,\dots, m_N$, initial positions $x_1,\dots, x_N$, and initial velocities $v_0(x_1),\dots, v_0(x_N)$. Then 
\be
\eta=\sum^N_{i=1}m_i\delta_{\gamma_i}\in {\cal P}(\Gamma)
\ee
satisfies properties $(i)-(vii)$ of Theorem \eqref{EtaThm}. Moreover, 
\be\label{PsiEstimate}
\int_{\Gamma}\Psi(\gamma)d\eta(\gamma)\le 
\int_\R\left(x^2+v_0(x)^2\right)d\rho_0(x)+\frac{1}{2}\int_\R\int_\R W'(x-y)^2d\rho_0(x)d\rho_0(y),
\ee
and
\begin{align}\label{PointwiseBound}
|\gamma(t)|&\le \left(\cosh(\sqrt{c}t)+\frac{\alpha}{\sqrt{c}}\sinh(\sqrt{c}t)\right)\left(|\gamma(0)|+\int_{\R}|x|d\rho_0(x)+1\right) \\
&\quad+\sqrt{\varphi(t)}\left( \int_\R\left(x^2+v_0(x)^2\right)d\rho_0(x)+\frac{1}{2}\int_\R\int_\R W'(x-y)^2d\rho_0(x)d\rho_0(y)\right)^{1/2}
\end{align}
for each $t\ge 0$ and $\gamma\in\textup{supp}(\eta)$. 
\end{lem}
\begin{proof}
\par 1. By \eqref{DiscreteEnergyEstA},  
\be
\int_\Gamma\left(\int^n_0\dot\gamma(t)^2dt\right)d\eta(\gamma)\le \varphi(n)\left(\int_\R v_0(x)^2d\rho_0(x)+\frac{1}{2}\int_\R\int_\R W'(x-y)^2d\rho_0(x)d\rho_0(y)\right)
\ee
for each $n\in \N$. As $\varphi(n)\ge 1$ for all $n\in \N$ and $\sum_{n\in\N}1/2^{n}=1$, it follows that 
\begin{align*}
\int_{\Gamma}\Psi(\gamma)d\eta(\gamma)
&=\int_{\Gamma}\left[\sum^\infty_{n=1}\frac{1}{2^n\varphi(n)}\left(\int^n_0\dot\gamma(t)^2dt+\gamma(0)^2\right)\right]d\eta(\gamma)\nonumber\\
&=\sum^\infty_{n=1}\frac{1}{2^n\varphi(n)}\int_{\Gamma}\left(\int^n_0\dot\gamma(t)^2dt\right)d\eta(\gamma)+
\sum^\infty_{n=1}\frac{1}{2^n\varphi(n)}\left(\int_{\Gamma}\gamma(0)^2d\eta(\gamma)\right)\nonumber\\
&\le \sum^\infty_{n=1}\frac{1}{2^n}\left(\int_\R\left(x^2+v_0(x)^2\right)d\rho_0(x)+\frac{1}{2}\int_\R\int_\R W'(x-y)^2d\rho_0(x)d\rho_0(y)\right)\\&=\int_\R\left(x^2+v_0(x)^2\right)d\rho_0(x)+\frac{1}{2}\int_\R\int_\R W'(x-y)^2d\rho_0(x)d\rho_0(y).
\end{align*}
This proves \eqref{PsiEstimate} and that $\eta$ satisfies property $(i)$ of Theorem \eqref{EtaThm}.

\par 2.  That $\eta$ satisfies property $(ii)$ follows from 
$$
\int_{\R}gd\rho_0=\sum^N_{i=1}m_ig(x_i)=\sum^N_{i=1}m_ig(\gamma_i(0))=\int_{\Gamma}g(\gamma(0))d\eta(\gamma).
$$
Properties $(iii)$ and $(iv)$ are a consequence of Proposition \ref{PropQSPP}. As for property $(v)$, set
$$
v(x,t)=
\begin{cases}
\dot\gamma_i(t+), \quad &x=\gamma_i(t)\\
0, \quad &\text{otherwise},
\end{cases}
$$
and note that part $(iii)$ of Proposition \ref{StickyParticlesExist} implies that $v$ is well defined. Property $(vi)$ is a corollary of \eqref{AveragingProp}, 
and property $(vii)$ follows from \eqref{ContNonincreaseEnergy}.

\par 3. As $\eta$ satisfies property $(vi)$, 
$$
|\gamma(t)-\xi(t)|\le\cosh(\sqrt{c}t)|\gamma(0)-\xi(0)| +\frac{1}{\sqrt{c}}\sinh(\sqrt{c}t)\omega\left(|\gamma(0)-\xi(0)|\right)
$$
for $\gamma,\xi\in\text{supp}(\eta)$ and $t\ge 0$. It follows that 
\begin{align}\label{TowardsPointwiseBound}
|\gamma(t)|&\le \int_{\Gamma}|\gamma(t)-\xi(t)|d\eta(\xi)+ \int_{\Gamma}|\xi(t)|d\eta(\xi)\nonumber\\
&\le \cosh(\sqrt{c}t) \int_{\R}|\gamma(0)-x|d\rho(x)+\frac{1}{\sqrt{c}}\sinh(\sqrt{c}t) \int_{\R}\omega(|\gamma(0)-x|)d\rho(x)\nonumber\\
&\quad+ \left(\int_{\Gamma}\xi(t)^2d\eta(\xi)\right)^{1/2}\nonumber\\
&\le \cosh(\sqrt{c}t) \left(|\gamma(0)|+\int_{\R}|x|d\rho(x)\right)+\frac{\alpha}{\sqrt{c}}\sinh(\sqrt{c}t)\left(1+|\gamma(0)|+\int_{\R}|x|d\rho(x)\right)\nonumber\\
&\quad+\sqrt{2}\left(\int_{\Gamma}\left(\xi(0)^2+\int^t_0\dot\xi(s)^2ds \right)d\eta(\xi)\right)^{1/2}\nonumber\\
&\le \left(\cosh(\sqrt{c}t)+\frac{\alpha}{\sqrt{c}}\sinh(\sqrt{c}t)\right)\left(1+|\gamma(0)|+\int_{\R}|x|d\rho(x)\right)\nonumber\\
&\quad+\sqrt{2\varphi(t)}\left( \int_\R\left(x^2+v_0(x)^2\right)d\rho_0(x)+\frac{1}{2}\int_\R\int_\R W'(x-y)^2d\rho_0(x)d\rho_0(y)\right)^{1/2}
\end{align}
which is \eqref{PointwiseBound}. 
\end{proof}
In the following proof of Theorem \ref{EtaThm}, we will say that $(\eta^k)_{k\in \N}\subset {\cal P}(\Gamma)$ {\it converges narrowly} to $\eta \in {\cal P}(\Gamma)$ provided 
\be\label{NarrowConvCondEta}
\lim_{k\rightarrow\infty}\int_{\Gamma}F(\gamma)d\eta^{k}(\gamma)=\int_{\Gamma}F(\gamma)d\eta(\gamma)
\ee
for each bounded, continuous $F: \Gamma\rightarrow\R.$   If $(\eta^k)_{k\in \N}\subset {\cal P}(\Gamma)$ converges narrowly to $\eta$ and
the limit \eqref{NarrowConvCondEta} exists and is finite for a continuous $F: \Gamma\rightarrow[0,\infty)$, we will say that $F$ is {\it uniformly integrable} (with respect to $(\eta^k) _{k\in \N}$). In particular, we will make use of the fact that 
if $F$ is uniformly integrable and $G: \Gamma\rightarrow\R$ is continuous with $|G|\le F$ on $\text{supp}(\eta^k)$ for each $k\in \N$, then $G$ is uniformly integrable, as well (this follows from Lemma 5.1.7 and the more general definition of uniformly integrability given in section 5.1.1 of \cite{AGS}). 

\par We also note that if $(\eta^k)_{k\in \N}\subset {\cal P}(\Gamma)$ converges narrowly to $\eta$, then 
\be\label{NarrowConvCondEtaSq}
\lim_{k\rightarrow\infty}\int_{\Gamma}H(\gamma,\xi)d\eta^{k}(\gamma)d\eta^{k}(\xi)=\int_{\Gamma}H(\gamma,\xi)d\eta(\gamma)d\eta(\xi)
\ee
bounded, continuous $H: \Gamma\times \Gamma\rightarrow\R$ (Theorem 2.8 \cite{Billingsley}).  In this case, we'll say $(\eta^k\times \eta^k)_{k\in \N} \subset {\cal P}(\Gamma\times \Gamma)$ converges narrowly to $\eta\times \eta$.  The notion of uniform integrability analogously extends to narrow convergence on ${\cal P}(\Gamma\times \Gamma)$.  

\begin{proof}[Proof of Theorem \ref{EtaThm}] Suppose $\rho_0\in {\cal P}(\R)$ satisfies \eqref{secondMoment}. We may select a sequence $(\rho^k_0)_{k\in \N}\subset {\cal P}(\R)$ in which each $\rho_0^k$ is a convex combination of Dirac measures, $\rho^k_0\rightarrow \rho_0$ narrowly, and 
\be\label{SecondMomentLimit}
\lim_{k\rightarrow\infty}\int_{\R}x^2d\rho^k_0(x)=\int_{\R}x^2d\rho_0(x).
\ee
The existence of such an approximating sequence is well known and can be verified as in Appendix A of \cite{Hynd}. Note that \eqref{SecondMomentLimit} and  assumption \eqref{Wprime} allow us to choose $B$ such that 
$$
 \int_\R\left(x^2+v_0(x)^2\right)d\rho^k_0(x)+\frac{1}{2}\int_\R\int_\R W'(x-y)^2d\rho^k_0(x)d\rho^k_0(y)\le B
$$
for $k\in \N$. 

\par By Lemma \ref{EtaLem}, there is an $\eta^k\in {\cal P}(\Gamma)$ satisfying conditions $(i)-(vii)$ with $\rho^k_0$ instead of $\rho_0$ for each $k\in \N$.  In view of \eqref{PsiEstimate} and our selection of $B$, 
$$
\sup_{k\in \N}\int_{\Gamma}\Psi(\gamma)d\eta^k(\gamma)\le B.
$$
As the sublevel sets of $\Psi$ are compact, Prokhorov's theorem (Theorem 5.1.3 of \cite{AGS}) implies there is a subsequence $(\eta^{k_j})_{j\in\N}$ which converges narrowly to some $\eta^\infty\in {\cal P}(\Gamma)$.
\par In view of \eqref{SecondMomentLimit}, $\gamma\mapsto\gamma(0)^2$ is uniformly integrable with respect to $(\eta^{k_j})_{j\in\N}$. By \eqref{PointwiseBound}, 
\be\label{PointwiseBoundTrue} 
\gamma(t)^2\le 4\left(\cosh(\sqrt{c}t)+\frac{\alpha}{\sqrt{c}}\sinh(\sqrt{c}t)\right)^2\left(\gamma(0)^2+B+1\right)+2\varphi(t)B
\ee
for $t\ge 0$ and each $\gamma\in \text{supp}(\eta^k)$.  It follows that $\gamma\mapsto\gamma(t)^2$ is also uniformly integrable with respect to $(\eta^{k_j})_{j\in\N}$ for each $t\ge 0$.  

\par We now proceed to show that $\eta^\infty$ has the claimed properties $(i)-(vii)$.  Whenever necessary, we will use that $\eta^{k_j}$ fulfills these conditions with $\rho^{k_j}_0$ instead of $\rho_0$ for each $j\in \N$. 
\\\\
\underline{Proof of $(i)$}.  Recall that $\Psi$ is lower semicontinuous and nonnegative. By the narrow convergence of $\eta^{k_j}\rightarrow \eta^\infty$, 
\be
\int_{\Gamma}\Psi(\gamma)d\eta^\infty(\gamma)\le \liminf_{j\rightarrow\infty}\int_{\Gamma}\Psi(\gamma)d\eta^{k_j}(\gamma)\le B
\ee
(Lemma 5.1.7 \cite{AGS}). Therefore, $\Psi(\gamma)<\infty$ for $\eta$ almost every $\gamma$. 
\\\\
\underline{Proof of $(ii)$}. 
Since $e_0:\Gamma\rightarrow\R;\gamma\mapsto \gamma(0)$ is continuous, 
$$
\rho_0=\lim_{j\rightarrow\infty}\rho^{k_j}_0=\lim_{j\rightarrow\infty}e_0{_\#}\eta^{k_j}=e_0{_\#}\eta^{\infty}.
$$
\\
\underline{Proof of $(iii)$ and $(iv)$}. For each $\gamma,\xi\in\text{supp}(\eta^\infty)$, there are $\gamma^j,\xi^j\in\text{supp}(\eta^{k_j})$
such that $\gamma^j\rightarrow \gamma$ and $\xi^j\rightarrow \xi$ as $j\rightarrow \infty$. For $0<s\le t$,
\begin{align*}
\frac{|\gamma(t)-\xi(t)|}{\sinh(\sqrt{c}t)}&=\lim_{j\rightarrow\infty} \frac{|\gamma^j(t)-\xi^j(t)|}{\sinh(\sqrt{c}t)}\\
&\le \lim_{j\rightarrow\infty} \frac{|\gamma^j(s)-\xi^j(s)|}{\sinh(\sqrt{c}s)}\\
&=\frac{|\gamma(s)-\xi(s)|}{\sinh(\sqrt{c}s)}.
\end{align*}
We can argue in the same way to establish $(iv)$. 
\\\\
\underline{Proof of $(v)$}. For $x\in \R$ and $0<s\le t$, set 
$$
f(x,t,s):=\inf\left\{\xi(t)+\frac{\sinh(\sqrt{c}t)}{\sinh(\sqrt{c}s)}|x-\xi(s)|: \xi\in\text{supp}(\eta^\infty)\right\}.
$$
Note that if $x=\gamma(s)$ for some $\gamma\in\text{supp}(\eta^\infty)$, we can choose $\xi=\gamma$ in the above 
infimum to get $f(\gamma(s),t,s)\le \gamma(t)$. By part $(iii)$, we also have that $\xi(t)+\frac{\sinh(\sqrt{c}t)}{\sinh(\sqrt{c}s)}|\gamma(s)-\xi(s)|
\ge \gamma(t)$ for all $\xi\in\text{supp}(\eta^\infty)$. As a result, 
$$
f(\gamma(s),t,s)= \gamma(t)
$$
for each $\gamma\in\text{supp}(\eta^\infty)$ and $s\le t$.  

\par Next we set
$$
v_n(x,t)=n(f(x,t+1/n,t)-x)
$$
for $(x,t)\in \R\times (0,\infty)$ and $n\in \N$.  Notice that $v_n$ is Borel measurable since $f$ is upper semicontinuous. Furthermore, 
\be\label{vndiffquotient}
v_n(\gamma(t),t)=n(\gamma(t+1/n)-\gamma(t))
\ee
for each $\gamma\in\text{supp}(\eta^\infty)$. 

\par In order to study the limit of $v_n$ as $n\rightarrow\infty$, we define
\begin{align*}
F&=\{(\gamma(t),t)\in\R\times(0,\infty): \gamma\in\text{supp}(\eta^\infty),\;\Psi(\gamma)<\infty,\dot\gamma(t)\;\text{exists} \}.
\end{align*}
Observe that 
$$
F=\bigcup_{m\in \N}\bigcap_{k\in \N}\bigcup_{N\in \N}F_{m,k,N},
$$
where 
\begin{align}
&F_{m,k,N}:=\displaystyle\bigcap_{\epsilon,\delta\in[-1/N,1/N]}\bigg\{(\gamma(t),t)\in\R\times[1/N,\infty): \gamma\in\text{supp}(\eta^\infty),\;\Psi(\gamma)\le m,\\
&\hspace{2in}\left. \left|\frac{\gamma(t+\epsilon)-\gamma(t)}{\epsilon}-\frac{\gamma(t+\delta)-\gamma(t)}{\delta}\right|\le \frac{1}{k}  \right\}.
\end{align}
As $\Psi$ has compact sublevel sets and $\text{supp}(\eta^\infty)$ is closed,
 it is straightforward to verify that each $F_{m,k,N}\subset \R\times (0,\infty)$ is the intersection of closed sets and is thus closed. Consequently, $F$
is a Borel subset of $\R\times (0,\infty)$.

\par In view of \eqref{vndiffquotient}, the limit
\be\label{derivativeLimitDub}
\dot\gamma(t)=\lim_{n\rightarrow \infty}v_n(\gamma(t),t)=:v(\gamma(t),t)
\ee
exists for each $(\gamma(t),t)\in F$.  Observe that this limit does not depend on $\gamma$. Indeed, if 
$(\xi(t),t)=(\gamma(t),t)\in F$, then $\xi(t+1/n)=\gamma(t+1/n)$ for $n\in \N$ by part $(iii)$ as $\xi,\gamma\in \text{supp}(\eta^\infty)$. Consequently
$$
v(\gamma(t),t)=\lim_{n\rightarrow \infty}n(\gamma(t+1/n)-\gamma(t))=\lim_{n\rightarrow \infty}n(\xi(t+1/n)-\xi(t))=v(\xi(t),t),
$$
so $v: F\rightarrow \R$ is well defined. 

\par  We emphasize that 
$$
v(x,t)=\lim_{n\rightarrow \infty}v_n(x,t)
$$
for $(x,t)\in F$. In particular, $v$ is Borel measurable as it is the pointwise limit of Borel functions. 
We may also extend $v$ to $\R\times (0,\infty)$ by setting it equal to zero on the complement of $F$. Once we identify $v$ with this extension, we obtain a Borel $v: \R\times(0,\infty)$ such that 
$$
\dot\gamma(t)=v(\gamma(t),t),\; a.e.\; t>0
$$
for $\gamma\in \text{supp}(\eta^\infty)$ with $\Psi(\gamma)<\infty$.
\\\\
\underline{Proof of $(vi)$}.  Suppose $h: \R\rightarrow \R$ is continuous with 
\be\label{hLinearGrowth}
|h(x)|\le C_0(1+|x|), \quad x\in \R.
\ee
We note that if $\eta^\infty$ satisfies property $(vi)$ for all $h$ satisfying \eqref{hLinearGrowth}, then $\eta^\infty$ satisfies property $(vi)$ for all Borel $h: \R\rightarrow \R$ with $\int_{\Gamma}h(\gamma(t))^2d\eta^\infty(\gamma)<\infty$ (Proposition 7.9 of \cite{Folland}). Consequently, it suffices to send $j\rightarrow\infty$ in
\begin{align}\label{EtaConsMomjay}
&\int^{t_2}_{t_1}\left(\int_{\Gamma}\dot\gamma(t)h(\gamma(t))d\eta^{k_j}(\gamma)\right)dt=  \\
&\hspace{1in}\int^{t_2}_{t_1}\left(\int_{\Gamma}\left(v_0(\gamma(0)) -\int^t_0W'*(e_s{_\#}\eta^{k_j})(\gamma(s))ds\right)h(\gamma(t))d\eta^{k_j}(\gamma)\right)dt
\end{align}
for $h$ which satisfies \eqref{hLinearGrowth} and each $0\le t_1\le t_2$.  

\par To this end, we first set $g(x):=\int^x_0h(y)dy$ and note 
$$
\int^{t_2}_{t_1}\left(\int_{\Gamma}\dot\gamma(t)h(\gamma(t))d\eta^{k_j}(\gamma)\right)dt=\int_{\Gamma}(g(\gamma(t_2))-g(\gamma(t_1)))d\eta^{k_j}(\gamma).
$$
In view of \eqref{hLinearGrowth}, $|g(\gamma(t))|$ grows at most quadratically in $|\gamma(t)|$. As $\gamma\mapsto \gamma(t)^2$ is uniformly integrable, $g(\gamma(t))$ is also uniformly integrable and 
\begin{align}\label{part6no1}
\lim_{j\rightarrow\infty}\int^{t_2}_{t_1}\left(\int_{\Gamma}\dot\gamma(t)h(\gamma(t))d\eta^{k_j}(\gamma)\right)dt&=\lim_{j\rightarrow\infty}\int_{\Gamma}(g(\gamma(t_2))-g(\gamma(t_1)))d\eta^{k_j}(\gamma)\nonumber\\
&=\int_{\Gamma}(g(\gamma(t_2))-g(\gamma(t_1)))d\eta^{\infty}(\gamma)\nonumber\\
&=\int^{t_2}_{t_1}\left(\int_{\Gamma}\dot\gamma(t)h(\gamma(t))d\eta^{\infty}(\gamma)\right)dt.
\end{align}
\par Next we note that since $v_0$ is absolutely continuous, we can choose a constant $C_1$ such that 
\be
|v_0(x)|\le C_1 (1+|x|)
\ee
Thus
\begin{align}\label{vzeroUpper}
|v_0(\gamma(0))h(\gamma(t))|&\le C_0(1+|\gamma(t)|)C_1(1+|\gamma(0)|)\nonumber\\
&\le \frac{1}{2}C_0^2(1+|\gamma(t)|)^2+\frac{1}{2}C_1(1+|\gamma(0)|)^2\nonumber\\
&\le C_0^2(1+\gamma(t)^2)+C_1^2(1+\gamma(0)^2).
\end{align}
Consequently, $v_0(\gamma(0))h(\gamma(t))$ uniformly integrable. Therefore,
\be\label{part7no1}
\lim_{j\rightarrow\infty}\int_{\Gamma} v_0(\gamma(0))h(\gamma(t))d\eta^{k_j}(\gamma)=\int_{\Gamma} v_0(\gamma(0))h(\gamma(t))d\eta^{\infty}(\gamma).
\ee
for all $t\ge 0.$ In view of \eqref{PointwiseBoundTrue} and \eqref{vzeroUpper}, $\int_{\Gamma} v_0(\gamma(0))h(\gamma(t))d\eta^{k_j}(\gamma)$ is uniformly bounded for $t\in [t_1,t_2]$ and $j\in \N$; so we can apply dominated convergence to find
\be\label{part6no2}
\lim_{j\rightarrow\infty}\int^{t_2}_{t_1}\left(\int_{\Gamma} v_0(\gamma(0))h(\gamma(t))d\eta^{k_j}(\gamma)\right)dt=\int^{t_2}_{t_1}\left(\int_{\Gamma} v_0(\gamma(0))h(\gamma(t))d\eta^{\infty}(\gamma)\right)dt.
\ee 

\par By \eqref{hLinearGrowth} and the at most linear growth of $W'$, there is a constant $C_2$ such that
\begin{align}\label{hWprimeUpper}
|h(\gamma(t))W'(\gamma(s)-\xi(s))|&\le C_0(1+|\gamma(t)|)|C_2(1+|\gamma(s)-\xi(s)|)\nonumber\\
& \le C_0^2(1+\gamma(t)^2)+C_2^2(1+(\gamma(s)-\xi(s))^2)\nonumber \\
& \le C_0^2(1+\gamma(t)^2)+2C_2^2(1+\gamma(s)^2+(\xi(s))^2).
\end{align}
As a result $(\gamma,\xi)\mapsto h(\gamma(t))W'(\gamma(s)-\xi(s))$
is uniformly integrable with respect to $(\eta^{k_j}\times\eta^{k_j})_{j\in \N}$.  It follows that 
\begin{align}
\lim_{j\rightarrow\infty}\int_{\Gamma}\int_{\Gamma}h(\gamma(t))W'(\gamma(s)-\xi(s))d\eta^{k_j}(\gamma)d\eta^{k_j}(\xi)=\int_{\Gamma}\int_{\Gamma}h(\gamma(t))W'(\gamma(s)-\xi(s))d\eta^{\infty}(\gamma)d\eta^{\infty}(\xi).
\end{align}
\par Combining \eqref{PointwiseBoundTrue} with \eqref{hWprimeUpper}, we see $\int_{\Gamma}\int_{\Gamma}h(\gamma(t))W'(\gamma(s)-\xi(s))d\eta^{k_j}(\gamma)d\eta^{k_j}(\xi)$ is uniformly bounded for $j\in \N$ and $s\in [0,t]$. By dominated convergence,
\begin{align}\label{part7no2}
&\lim_{j\rightarrow\infty}\int_{\Gamma}\left(\int^t_0W'*(e_s{_\#}\eta^{k_j})(\gamma(s))ds\right)h(\gamma(t))d\eta^{k_j}(\gamma)\nonumber\\
&\hspace{1in}=\lim_{j\rightarrow\infty}\int^t_0\left(\int_{\Gamma}\int_{\Gamma}h(\gamma(t))W'(\gamma(s)-\xi(s))d\eta^{k_j}(\gamma)d\eta^{k_j}(\xi)\right)ds\nonumber\\
&\hspace{1in}=\int^t_0\left(\int_{\Gamma}\int_{\Gamma}h(\gamma(t))W'(\gamma(s)-\xi(s))d\eta^{\infty}(\gamma)d\eta^{\infty}(\xi)\right)ds\nonumber\\
&\hspace{1in} =\int_{\Gamma}\left(\int^t_0W'*(e_s{_\#}\eta^{\infty})(\gamma(s))ds\right)h(\gamma(t))d\eta^{\infty}(\gamma)
\end{align}
for each $t\ge 0$. 
In a very similar way, we conclude 
\begin{align}\label{part6no3}
&\lim_{j\rightarrow\infty}\int^{t_2}_{t_1}\int_{\Gamma}\left(\int^t_0W'*(e_s{_\#}\eta^{k_j})(\gamma(s))ds\right)h(\gamma(t))d\eta^{k_j}(\gamma)dt\\
&\hspace{1in}=\int^{t_2}_{t_1}\int_{\Gamma}\left(\int^t_0W'*(e_s{_\#}\eta^{\infty})(\gamma(s))ds\right)h(\gamma(t))d\eta^{\infty}(\gamma)dt.
\end{align}
Putting this limit together with \eqref{part6no1} and \eqref{part6no2} allows us to send $j\rightarrow\infty$ in \eqref{EtaConsMomjay} and find 
\begin{align}
&\int^{t_2}_{t_1}\left(\int_{\Gamma}\dot\gamma(t)h(\gamma(t))d\eta^{\infty}(\gamma)\right)dt\\
&\hspace{1in} =\int^{t_2}_{t_1}\left(\int_{\Gamma}\left(v_0(\gamma(0)) -\int^t_0W'*(e_s{_\#}\eta^{\infty})(\gamma(s))ds\right)h(\gamma(t))d\eta^{\infty}(\gamma)\right)dt.
\end{align}
\\\\
\underline{Proof of $(vii)$}.  Let us set 
$$
E^j(t):=\int_{\Gamma}\frac{1}{2}\dot\gamma(t+)^2d\eta^{k_j}(\gamma)+\frac{1}{2}\int_{\Gamma}\int_{\Gamma}W(\gamma(t)-\xi(t))d\eta^{k_j}(\gamma)d\eta^{k_j}(\xi)
$$
for $j\in \N$ and every $t\ge 0$. We will show 
\be\label{PointwiseBoundEj}
\sup_{j\in \N}\max_{t\in [t_1,t_2]}|E^j(t)|<\infty
\ee
and
\be\label{WeakConvEj}
\lim_{j\rightarrow\infty}\int^{t_2}_{t_1}E^j(t)dt= \int^{t_2}_{t_1}E^\infty(t)dt
\ee
for each $t_1\le t_2$, where 
$$
E^\infty(t):=\int_{\Gamma}\frac{1}{2}\dot\gamma(t)^2d\eta^{\infty}(\gamma)+\frac{1}{2}\int_{\Gamma}\int_{\Gamma}W(\gamma(t)-\xi(t))d\eta^{\infty}(\gamma)d\eta^{\infty}(\xi)
$$
for almost every $t\ge 0$. As $(E^j)_{j\in \N}$ is a sequence of nonincreasing functions (by \eqref{ContNonincreaseEnergy}) which is uniformly bounded on each compact subinterval of $[0,\infty)$, Helly's selection theorem implies $\overline{E}(t):=\lim_{j\rightarrow\infty}E^j(t)$ exists for all $t\ge 0$.  Clearly $\overline{E}$ is nonincreasing. By \eqref{WeakConvEj}, $E^\infty(t)=\overline{E}(t)$ for almost every $t\ge 0$.  We then would conclude that $\eta^\infty$ satisfies $(vii)$.

\par In order prove \eqref{PointwiseBoundEj}, we first note that since $W$ is semiconvex and $W'$ grows at most linearly, $W$ grows at most quadratically. In particular, there is a constant $C_3\ge 0$ for which
\be\label{WxyBound}
|W(x-y)|\le C_3(1+x^2+y^2).
\ee
Thus 
\begin{align*}
E^j(t)&\le E^j(0)\\
&= \int_{\R}\frac{1}{2}v_0(x)^2d\rho^k_0(x)+\frac{1}{2}\int_{\R}\int_{\R}W(x-y)d\rho^k_0(x)d\rho^k_0(y)\\
&\le B+\frac{C_3}{2}\int_{\R}\int_{\R}(1+x^2+y^2)d\rho^k_0(x)d\rho^k_0(y)\\
&\le B+C_3\left(\frac{1}{2}+\int_{\R}x^2d\rho^k_0(x)\right)\\
&\le B+C_3(1/2+B)
\end{align*}
for each $t\ge 0$.  Moreover, 
\begin{align*}
E^j(t)&\ge \frac{1}{2}\int_{\Gamma}\int_{\Gamma}W(\gamma(t)-\xi(t))d\eta^{k_j}(\gamma)d\eta^{k_j}(\xi)\\
&\ge -\frac{C_3}{2}\int_{\Gamma}\int_{\Gamma}(1+\gamma(t)^2+\xi(t)^2)d\eta^{k_j}(\gamma)d\eta^{k_j}(\xi)\\
&\ge -C_3\left(\frac{1}{2}+\int_{\Gamma}\gamma(t)^2d\eta^{k_j}(\gamma)\right)
\end{align*}
In view of \eqref{PointwiseBoundTrue}, $\int_{\Gamma}\gamma(t)^2d\eta^{k_j}(\gamma)$ is bounded above independently of 
$j\in \N$ and $t\in[t_1,t_2]$.  These upper and lower bounds together prove \eqref{PointwiseBoundEj}. 

\par In order to show \eqref{WeakConvEj}, we note $\gamma\mapsto \gamma(t)^2$ is uniformly integrable and $|W(\gamma(t)-\xi(t))|\le C_3(1+\gamma(t)^2+\xi(t)^2)$. It follows that
$$
\lim_{j\rightarrow\infty}\int_{\Gamma}\int_{\Gamma}W(\gamma(t)-\xi(t))d\eta^{k_j}(\gamma)d\eta^{k_j}(\xi)
=\int_{\Gamma}\int_{\Gamma}W(\gamma(t)-\xi(t))d\eta(\gamma)d\eta(\xi).
$$
It is also straightforward to combine \eqref{PointwiseBoundTrue} with \eqref{WxyBound} to show that the function $t\mapsto \int_{\Gamma}\int_{\Gamma}W(\gamma(t)-\xi(t))d\eta^{k_j}(\gamma)d\eta^{k_j}(\xi)$ is uniformly bounded for on the interval $[t_1,t_2]$. Dominated convergence then implies 
\be\label{ConvInteractionPot}
\lim_{j\rightarrow\infty}\int^{t_2}_{t_1}\int_{\Gamma}\int_{\Gamma}W(\gamma(t)-\xi(t))d\eta^{k_j}(\gamma)d\eta^{k_j}(\xi)dt
=\int^{t_2}_{t_1}\int_{\Gamma}\int_{\Gamma}W(\gamma(t)-\xi(t))d\eta(\gamma)d\eta(\xi)dt.
\ee

\par Also observe that since $\eta^{k_j}$ satisfy property $(v)$ and $(vi)$,
\begin{align*}
&\int_{\Gamma}\left(\int^{t_2}_{t_1}\dot\gamma(t)^2dt\right)d\eta^{k_j}(\gamma)\\
&\quad=\int^{t_2}_{t_1}\left(\int_{\Gamma}\dot\gamma(t)^2d\eta^{k_j}(\gamma)\right)dt\\
&\quad=\int^{t_2}_{t_1}\left(\int_{\Gamma}\left(v_0(\gamma(0)) -\int^t_0W'*(e_s{_\#}\eta^{k_j})(\gamma(s))ds\right)\dot\gamma(t)d\eta^{k_j}(\gamma)\right)dt\\
&\quad=\int_{\Gamma}\left(\int^{t_2}_{t_1}\left(v_0(\gamma(0)) -\int^t_0W'*(e_s{_\#}\eta^{k_j})(\gamma(s))ds\right)\dot\gamma(t)dt\right)d\eta^{k_j}(\gamma)\\
&\quad=\int_{\Gamma}\left\{\left.\left(v_0(\gamma(0)) -\int^t_0W'*(e_s{_\#}\eta^{k_j})(\gamma(s))ds\right)\gamma(t)\right|^{t_2}_{t_1}\right.\\
&\hspace{2in} \left.+\int^{t_2}_{t_1}\gamma(t)W'*(e_t{_\#}\eta^{k_j})(\gamma(t))dt\right\}d\eta^{k_j}(\gamma)
\end{align*} 
The limit
\be
\lim_{j\rightarrow\infty}\left.\int_{\Gamma}v_0(\gamma(0))\gamma(t)\right|^{t_2}_{t_1}d\eta^{k_j}(\gamma)=
\left.\int_{\Gamma}v_0(\gamma(0))\gamma(t)\right|^{t_2}_{t_1}d\eta^{\infty}(\gamma)
\ee
follows from \eqref{part7no1} while 
\begin{align*}
&\lim_{j\rightarrow\infty}\left.\int_{\Gamma}\left(\int^t_0W'*(e_s{_\#}\eta^{k_j})(\gamma(s))ds\right)\gamma(t)\right|^{t_2}_{t_1}d\eta^{k_j}(\gamma)\\
&\hspace{1in}=
\left.\int_{\Gamma}\left(\int^t_0W'*(e_s{_\#}\eta^{\infty})(\gamma(s))ds\right)\gamma(t)\right|^{t_2}_{t_1}d\eta^{\infty}(\gamma)
\end{align*}
in a consequence of \eqref{part7no2}. Once we write, 
\be
\int_{\Gamma}\left(\int^{t_2}_{t_1}\gamma(t)W'*(e_t{_\#}\eta^{k_j})(\gamma(t))dt\right)d\eta^{k_j}(\gamma)
=\int^{t_2}_{t_1}\int_{\Gamma}\int_{\Gamma}\gamma(t)W'(\gamma(t)-\xi(t))d\eta^{k_j}(\gamma)d\eta^{k_j}(\xi),
\ee
we see that the limit 
\begin{align*}
\lim_{j\rightarrow\infty}\int_{\Gamma}\left(\int^{t_2}_{t_1}\gamma(t)W'*(e_t{_\#}\eta^{k_j})(\gamma(t))dt\right)d\eta^{k_j}(\gamma)
=\int_{\Gamma}\left(\int^{t_2}_{t_1}\gamma(t)W'*(e_t{_\#}\eta^{\infty})(\gamma(t))dt\right)d\eta^{\infty}(\gamma)
\end{align*}
follows from a minor variation of our proof of \eqref{part7no2}. 

\par As a result, 
\begin{align}
&\lim_{j\rightarrow\infty}\int^{t_2}_{t_1}\left(\int_{\Gamma}\dot\gamma(t)^2d\eta^{k_j}(\gamma)\right)dt\\
&\quad=\lim_{j\rightarrow\infty}\int_{\Gamma}\left\{\left.\left(v_0(\gamma(0)) -\int^t_0W'*(e_s{_\#}\eta^{k_j})(\gamma(s))ds\right)\gamma(t)\right|^{t_2}_{t_1}\right.\\
&\hspace{2in} \left.+\int^{t_2}_{t_1}\gamma(t)W'*(e_t{_\#}\eta^{k_j})(\gamma(t))dt\right\}d\eta^{k_j}(\gamma)\\
&\quad=\int_{\Gamma}\left\{\left.\left(v_0(\gamma(0)) -\int^t_0W'*(e_s{_\#}\eta^{\infty})(\gamma(s))ds\right)\gamma(t)\right|^{t_2}_{t_1}\right.\\
&\hspace{2in} \left.+\int^{t_2}_{t_1}\gamma(t)W'*(e_t{_\#}\eta^{\infty})(\gamma(t))dt\right\}d\eta^{\infty}(\gamma)\\
&\quad=\int^{t_2}_{t_1}\left(\int_{\Gamma}\dot\gamma(t)^2d\eta^{\infty}(\gamma)\right)dt.
\end{align}
This limit combined with \eqref{ConvInteractionPot} implies \eqref{WeakConvEj}, as desired.
\end{proof}

\section{Solutions to the pressureless Euler system}\label{CompactSect}
This section is dedicated to the proof of Theorem \ref{mainThm}.  So we assume $\rho_0\in {\cal P}(\R)$ with $\int_{\R}x^2d\rho_0(x)<\infty$, $v_0:\R\rightarrow\R$ is absolutely continuous, $W(x)+(c/2)x^2$ is convex and $|W'(x)|$ grows at most linearly as $|x|\rightarrow\infty$.  
According to Theorem \ref{EtaThm}, there is $\eta\in {\cal P}(\Gamma)$ which satisfies parts $(i)-(vii)$ of that statement.  We will simply 
refer to these parts by their respective numbers below. 

\par Let us define 
$$
\rho: (0,\infty)\rightarrow {\cal P}(\R); t\mapsto e_t{_\#}\eta
$$
and select a Borel $v: \R\times (0,\infty)\rightarrow \R$ from part $(vi)$.  By $(i)$, $(ii)$, and $(v)$, 
\begin{align*}
\int^\infty_0\int_{\R}(\partial_t\phi+v\partial_x\phi )d\rho_tdt
&=\int^\infty_0 \int_{\Gamma}(\partial_t\phi(\gamma(t),t)+ v(\gamma(t),t)\partial_x\phi(\gamma(t),t))d\eta(\gamma)dt\\
&=\int^\infty_0 \int_{\Gamma}(\partial_t\phi(\gamma(t),t)+ \dot\gamma(t)\partial_x\phi(\gamma(t),t))d\eta(\gamma)dt\\
&=\int_{\Gamma}\int^\infty_0\frac{d}{dt}\phi(\gamma(t),t)dt d\eta(\gamma)\\
&=-\int_{\Gamma}\phi(\gamma(0),0)d\eta(\gamma)\\
&=-\int_{\R}\phi(\cdot,0)d\rho_0
\end{align*}
for any $\phi\in C^\infty_c(\R\times[0,\infty))$; and in view of $(vi)$,
\begin{align*}
&\int^\infty_0\int_\R(v\partial_t\phi+v^2\partial_x\phi)d\rho_tdt\\
& =\int^\infty_0\int_{\Gamma}\left[\partial_t\phi(\gamma(t),t)+v(\gamma(t),t)\partial_x\phi(\gamma(t),t)\right]v(\gamma(t),t)d\eta(\gamma)dt\\
& =\int^\infty_0\int_{\Gamma}\underbrace{\left[\partial_t\phi(\gamma(t),t)+v(\gamma(t),t)\partial_x\phi(\gamma(t),t)\right]}_{\frac{d}{dt}\phi(\gamma(t),t)}\dot\gamma(t)d\eta(\gamma)dt\\
& =\int^\infty_0\int_{\Gamma}\frac{d}{dt}\phi(\gamma(t),t)\left[v_0(\gamma(0)) -\int^t_0W'*(e_{s}{_\#}\eta)(\gamma(s))ds   \right]d\eta(\gamma)dt\\
& =\int_{\Gamma}\int^\infty_0\frac{d}{dt}\phi(\gamma(t),t)\left[v_0(\gamma(0)) -\int^t_0W'*(e_{s}{_\#}\eta)(\gamma(s))ds   \right]dtd\eta(\gamma)\\
& =\int^\infty_0\int_{\Gamma}\phi(\gamma(t),t)W'*(e_{t}{_\#}\eta)(\gamma(t)) d\eta(\gamma)dt-\int_{\Gamma}\phi(\gamma(0),0)v_0(\gamma(0))d\eta(\gamma)\\
&=\int^\infty_0\int_\R \phi\left(W'*\rho_t\right)d\rho_tdt-\int_{\R}\phi(\cdot,0)v_0d\rho_0.
\end{align*}
As a result, $\rho$ and $v$ is a weak solution pair of the pressureless Euler equations which satisfies the initial conditions $\rho|_{t=0}=\rho_0$ and $v|_{t=0}=v_0$.

\par A consequence of $(v)$ is that for almost every $t>0$, 
\be\label{vODE}
\dot\gamma(t)=v(\gamma(t),t)
\ee
for $\eta$ almost every $\gamma\in \Gamma$.  Combining this with part $(vii)$ gives
\begin{align*}
&\frac{1}{2}\int_\R v(x,t)^2d\rho_t(x)+\frac{1}{2}\int_\R\int_\R W(x-y)d\rho_t(x)d\rho_t(y)\\
&\quad =\int_{\Gamma}\frac{1}{2}\dot\gamma(t)^2d\eta(\gamma)+\frac{1}{2}\int_{\Gamma}\int_{\Gamma}W(\gamma(t)-\xi(t))d\eta(\gamma)d\eta(\xi)\\
&\quad \le \int_{\Gamma}\frac{1}{2}\dot\gamma(s)^2d\eta(\gamma)+\frac{1}{2}\int_{\Gamma}\int_{\Gamma}W(\gamma(s)-\xi(s))d\eta(\gamma)d\eta(\xi)\\
&\quad = \frac{1}{2}\int_\R v(x,s)^2d\rho_s(x)+\frac{1}{2}\int_\R\int_\R W(x-y)d\rho_s(x)d\rho_s(y)
\end{align*}
for almost every $0\le s\le t$. This proves \eqref{EnergyIneq}.

\par Fix such a time $t>0$ and choose a Borel 
subset $S\subset \Gamma$ such that \eqref{vODE} holds for each $\gamma\in S$ and $\eta(S)=1$.  By $(iii)$, 
\be
\left.\frac{d}{ds}\frac{(\gamma(s)-\xi(s))^2}{s^2}\right|_{s=t}=\frac{2}{t^2}\left((v(\gamma(t),t)-v(\xi(t),t))(\gamma(t)-\xi(t))-\frac{1}{t}(\gamma(t)-\xi(t))^2\right)\le 0
\ee
for $\gamma,\xi\in S$.   Thus, 
\be\label{weakEntropy}
(v(x,t)-v(y,t))(x-y)\le \frac{1}{t}(x-y)^2,\quad x,y\in e_t(S).
\ee

For each $\epsilon>0$, 
we may also select a closed $F\subset S$ such that $\eta(S\setminus F)\le \epsilon$ (Theorem 1.1 \cite{Billingsley}).
Observe 
\be
e_t(F\cap\{\Psi<\infty\})=\bigcup^\infty_{m=1}e_t(F\cap\{\Psi\le m\} ).
\ee
Since $\Psi$ has compact sublevel sets, $F\cap\{\Psi\le m\}$ is compact and so $e_t(F\cap\{\Psi<\infty\})$ is Borel. As $\eta(\{\Psi<\infty\})=1$ and $\eta(F)\ge 1-\epsilon$, 
\be\label{rhoteeoneminuseps}
\rho_t(e_t(F\cap\{\Psi<\infty\}))=\eta\left(e_t^{-1}\left(e_t(F\cap\{\Psi<\infty\})\right)\right)\ge \eta\left(F\cap\{\Psi<\infty\}\right)\ge 1-\epsilon.
\ee
Combining this inequality with \eqref{weakEntropy} gives
\begin{align}
1 &\ge \left(\rho_t\times\rho_t\right)\left(\left\{(x,y)\in \R\times \R:(v(x,t)-v(y,t))(x-y)\le \frac{1}{t}(x-y)^2 \right\}\right)\\
   &\ge \left(\rho_t\times\rho_t\right)\left(e_t(F\cap\{\Psi<\infty\})\times e_t(F\cap\{\Psi<\infty\}) \right)\\
   &\ge \left(\rho_t(e_t(F\cap\{\Psi<\infty\}))\right)^2\\
   &\ge (1-\epsilon)^2.
\end{align}

\par Since $\epsilon>0$ is arbitrary,
\begin{align}
1&=\left(\rho_t\times\rho_t\right)\left(\left\{(x,y)\in \R\times \R:(v(x,t)-v(y,t))(x-y)\le \frac{1}{t}(x-y)^2 \right\}\right)\\
&=\int_{\R}\rho_t\left(\left\{x\in \R:(v(x,t)-v(y,t))(x-y)\le \frac{1}{t}(x-y)^2 \right\}\right)d\rho_t(y).
\end{align}
Thus, for $\rho_t$ almost every $y\in \R$
$$
\rho_t\left(\left\{x\in \R:(v(x,t)-v(y,t))(x-y)\le \frac{1}{t}(x-y)^2 \right\}\right)=1.
$$
That is, for $\rho_t$ almost every every $x,y\in \R$,
$$
(v(x,t)-v(y,t))(x-y)\le \frac{1}{t}(x-y)^2.
$$

\appendix

\section{Newton's equations}\label{NewtonSec}
Here we show that the ODE system \eqref{NewtonSystem} has a solution on the interval $[0,\infty)$ for prescribed initial conditions.   We recall the standing assumptions that $W:\R\rightarrow\R$ is continuously differentiable, $W$ is even, and that \eqref{Wsemiconvex} holds.  
\begin{prop}\label{ExistenceODE}
Suppose $m_1,\dots, m_N> 0$, $x_1,\dots, x_N\in \R$ and $v_1,\dots, v_N\in \R$. There are
$$
\gamma_1,\dots,\gamma_N\in C^2([0,\infty))
$$
satisfying 
\be\label{NewtonSystem2}
\ddot \gamma_i(t)=-\sum^N_{j=1}m_jW'(\gamma_i(t)-\gamma_j(t))
\ee
for  $t>0$ and
\be\label{NewtonSystemInit}
\gamma_i(0)=x_i\quad \text{and}\quad \dot\gamma_i(0+)=v_i
\ee
for $i=1,\dots, N$.  
\end{prop}

\begin{proof}
By Peano's existence theorem, there is a solution $\gamma_1,\dots,\gamma_N\in C^2([0,T))$ of the ODE \eqref{NewtonSystem2} for some $T\in (0,\infty]$ which satisfies the initial conditions \eqref{NewtonSystemInit}. We may assume that $[0,T)$ is the maximal interval of existence so that this solution cannot be continued to a larger interval if $T<\infty$. In this case, it must be that 
\be\label{GammaiCantHold}
\sup_{t\in [0,T)}|\dot \gamma_j(t)|=\infty
\ee
for some $j=1,\dots, N$. Otherwise, $\sup_{t\in [0,T)}| \dot\gamma_i(t)|<\infty$ and 
$$
\sup_{t\in [0,T)}| \gamma_i(t)|\le|x_i| +T \sup_{t\in [0,T)}| \dot\gamma_i(t)|<\infty
$$
for all $i=1,\dots, N$ and this solution $\gamma_1,\dots,\gamma_N$ could then be continued to $[0,T+\epsilon)$ for some $\epsilon>0$ (Chapter 1 of \cite{MR587488}). 

\par Observe 
\be\label{ConsEnergy}
\frac{1}{2}\sum^N_{i=1}m_i\dot\gamma_i(t)^2+
\frac{1}{2}\sum^N_{i,j=1}m_im_jW(\gamma_i(t)-\gamma_j(t))=\frac{1}{2}\sum^N_{i=1}m_i v_i^2+
\frac{1}{2}\sum^N_{i,j=1}m_im_jW(x_i-x_j)
\ee 
for $t\in [0,T)$. This can be verified by differentiating the left hand side of \eqref{ConsEnergy} and by using that $\gamma_1,\dots, \gamma_N$ solves \eqref{NewtonSystem2}.   Arguing as we did to prove Corollary \ref{discreteEnergyEstimates}, we find 
\be
\sum^N_{i=1}m_i\dot\gamma_i(t)^2ds\le \varphi'(t)\;
\left(\sum^N_{i=1}m_iv_0(x_i)^2 +\frac{1}{2}\sum^N_{i,j=1}m_im_jW'(x_i-x_j)^2 \right)
\ee
for $t\in [0,T)$. As $m_i>0$ for each $i=1,\dots, N$, \eqref{GammaiCantHold} could not hold for any $j=1,\dots, N$. We conclude that $T=\infty$. 

\end{proof}

\bibliography{EPbib}{}

\begin{thebibliography}{10}

\bibitem{AGS}
L.~Ambrosio, N.~Gigli, and G.~Savar\'e.
\newblock {\em Gradient flows in metric spaces and in the space of probability
  measures}.
\newblock Lectures in Mathematics ETH Z\"urich. Birkh\"auser Verlag, Basel,
  second edition, 2008.

\bibitem{Billingsley}
P.~Billingsley.
\newblock {\em Convergence of probability measures}.
\newblock Wiley Series in Probability and Statistics: Probability and
  Statistics. John Wiley \& Sons, Inc., New York, second edition, 1999.
\newblock A Wiley-Interscience Publication.

\bibitem{BreGan}
Y.~Brenier, W.~Gangbo, G.~Savar\'e, and M.~Westdickenberg.
\newblock Sticky particle dynamics with interactions.
\newblock {\em J. Math. Pures Appl. (9)}, 99(5):577--617, 2013.

\bibitem{BreGre}
Y.~Brenier and E.~Grenier.
\newblock Sticky particles and scalar conservation laws.
\newblock {\em SIAM J. Numer. Anal.}, 35(6):2317--2328, 1998.

\bibitem{MR3296602}
F.~Cavalletti, M.~Sedjro, and M.~Westdickenberg.
\newblock A simple proof of global existence for the 1{D} pressureless gas
  dynamics equations.
\newblock {\em SIAM J. Math. Anal.}, 47(1):66--79, 2015.

\bibitem{ERykovSinai}
W.~E, Y.~Rykov, and Y.~Sinai.
\newblock Generalized variational principles, global weak solutions and
  behavior with random initial data for systems of conservation laws arising in
  adhesion particle dynamics.
\newblock {\em Comm. Math. Phys.}, 177(2):349--380, 1996.

\bibitem{Folland}
G.~Folland.
\newblock {\em Real analysis}.
\newblock Pure and Applied Mathematics (New York). John Wiley \& Sons, Inc.,
  New York, second edition, 1999.
\newblock Modern techniques and their applications, A Wiley-Interscience
  Publication.

\bibitem{GNT}
W.~Gangbo, T.~Nguyen, and A.~Tudorascu.
\newblock Euler-{P}oisson systems as action-minimizing paths in the
  {W}asserstein space.
\newblock {\em Arch. Ration. Mech. Anal.}, 192(3):419--452, 2009.

\bibitem{Guo}
Y.~Guo, L.~Han, and J.~Zhang.
\newblock Absence of shocks for one dimensional {E}uler-{P}oisson system.
\newblock {\em Arch. Ration. Mech. Anal.}, 223(3):1057--1121, 2017.

\bibitem{Gurbatov}
S.~N Gurbatov, A.~Saichev, and S.~F Shandarin.
\newblock Large-scale structure of the universe. the zeldovich approximation
  and the adhesion model.
\newblock {\em Physics-Uspekhi}, 55(3):223, 2012.

\bibitem{MR587488}
J.~Hale.
\newblock {\em Ordinary differential equations}.
\newblock Robert E. Krieger Publishing Co., Inc., Huntington, N.Y., second
  edition, 1980.

\bibitem{Hynd}
R.~Hynd.
\newblock Probability measures on the path space and the sticky particle
  system.
\newblock {\em Annali della Scuola Normale Superiore di Pisa, Classe di
  Scienze}, In Press.

\bibitem{Hynd2}
R.~Hynd.
\newblock A trajectory map for the pressureless euler equations.
\newblock {\em Transactions of the American Mathematical Society}, In Press.

\bibitem{MR3302526}
Ryan Hynd and Hwa~Kil Kim.
\newblock Infinite horizon value functions in the {W}asserstein spaces.
\newblock {\em J. Differential Equations}, 258(6):1933--1966, 2015.

\bibitem{Jabin}
P.-E. Jabin and T.~Rey.
\newblock Hydrodynamic limit of granular gases to pressureless {E}uler in
  dimension 1.
\newblock {\em Quart. Appl. Math.}, 75(1):155--179, 2017.

\bibitem{Jin}
C.~Jin.
\newblock Well posedness for pressureless {E}uler system with a flocking
  dissipation in {W}asserstein space.
\newblock {\em Nonlinear Anal.}, 128:412--422, 2015.

\bibitem{NatSav}
L.~Natile and G.~Savar\'e.
\newblock A {W}asserstein approach to the one-dimensional sticky particle
  system.
\newblock {\em SIAM J. Math. Anal.}, 41(4):1340--1365, 2009.

\bibitem{NguTud}
T.~Nguyen and A.~Tudorascu.
\newblock Pressureless {E}uler/{E}uler-{P}oisson systems via adhesion dynamics
  and scalar conservation laws.
\newblock {\em SIAM J. Math. Anal.}, 40(2):754--775, 2008.

\bibitem{MR3359159}
T~Nguyen and A.~Tudorascu.
\newblock One-dimensional pressureless gas systems with/without viscosity.
\newblock {\em Comm. Partial Differential Equations}, 40(9):1619--1665, 2015.

\bibitem{Shen}
C.~Shen.
\newblock The {R}iemann problem for the pressureless {E}uler system with the
  {C}oulomb-like friction term.
\newblock {\em IMA J. Appl. Math.}, 81(1):76--99, 2016.

\bibitem{Zeldovich}
Ya.~B. Zel'dovich.
\newblock {Gravitational instability: An Approximate theory for large density
  perturbations}.
\newblock {\em Astron. Astrophys.}, 5:84--89, 1970.

\end{thebibliography}

\bibliographystyle{plain}

\end{document}